\documentclass[12pt]{amsart}

\usepackage{amsmath,bm,amssymb,amsthm,textcomp}
\usepackage{enumitem}
\usepackage{mathtools}
\usepackage{hyperref}
\usepackage[numbers,sort&compress]{natbib}
\usepackage{tikz-cd}
\usetikzlibrary{cd}

\DeclareMathOperator{\diam}{diam}
\DeclareMathOperator{\hdim}{\dim_H}
\DeclareMathOperator{\bdim}{\dim_B}
\DeclareMathOperator{\Div}{div}

\theoremstyle{plain}
\newtheorem{theorem}{Theorem}[section]
\newtheorem{lemma}[theorem]{Lemma}
\newtheorem{proposition}[theorem]{Proposition}
\newtheorem{corollary}[theorem]{Corollary}
\newtheorem{fact}[theorem]{Fact}

\theoremstyle{definition}
\newtheorem{definition}[theorem]{Definition}

\theoremstyle{remark}
\newtheorem{remark}[theorem]{Remark}

\counterwithin{equation}{section}

\allowdisplaybreaks

\begin{document}

\title{Convergence exponent of Pierce expansion digit sequences}

\author{Min Woong Ahn}
\address{Department of Mathematics, SUNY at Buffalo, Buffalo, NY 14260-2900, USA}
\email{minwoong@buffalo.edu}
\curraddr{Department of Mathematics Education, Silla University, 140, Baegyang-daero 700beon-gil, Sasang-gu, Busan, 46958, Republic of Korea}
\email{minwoong@silla.ac.kr}

\date{\today}


\begin{abstract}
In this paper, we investigate the convergence exponent of Pierce expansion digit sequences. We explore some basic properties of the convergence exponent as a real-valued function defined on the closed unit interval, as well as those of the level sets of the function. Additionally, we further study subsets of the closed unit interval on which the series of positive $s$th powers of the reciprocals of the Pierce expansion digits diverges.
\end{abstract}

\maketitle

\tableofcontents

\section{Introduction} \label{Introduction}

P{\'o}lya and Szeg{\H o} \cite{PS98} introduced the concept of the {\em convergence exponent} of a non-decreasing sequence of positive real numbers diverging to infinity and defined the convergence exponent of a sequence $\sigma \coloneqq (\sigma_n)_{n \in \mathbb{N}}$ by
\[
\lambda (\sigma) \coloneqq \inf \left\{ s > 0 : \sum_{n \in \mathbb{N}} \frac{1}{\sigma_n^s} \text{ converges} \right\}.
\]
By definition, the convergence exponent $\lambda$ is a mapping from a certain sequence space to the set of non-negative extended real numbers, i.e., the interval $[0,\infty]$.

Kostyrko \cite{Kos79} investigated the convergence exponent $\lambda \colon S \to [0,\infty]$, where the domain $S$, equipped with the topology generated by the Fr{\'e}chet metric, consists of real non-decreasing sequences whose initial terms are bounded below by some fixed positive constant. Among others, the Baire class of the function $\lambda \colon S \to [0,\infty]$ and the Borel hierarchy and the Baire category of the superlevel set $\{ \sigma \in S : \lambda (\sigma) \geq \alpha \}$ and those of the sublevel set $\{ \sigma \in S : \lambda (\sigma) \leq \alpha \}$ for $\alpha \in \mathbb{R}$ were investigated. After a few years, Kostyrko and {\v S}al{\'a}t \cite{KS82} obtained more general results by removing the boundedness condition on the first term. A few more years later, {\v S}al{\'a}t \cite{Sal84} studied the convergence exponent of subsequences. More recently, in \cite{FMSW21} and \cite{SW21}, the authors studied the convergence exponent of digit sequences appearing in the continued fraction expansions and Engel expansions, respectively. 

In this paper, we are concerned with the convergence exponent of Pierce expansion digit sequences, i.e., the mapping $\lambda^* \colon [0,1] \to [0, \infty]$ defined by
\begin{align} \label{definition of lambda star}
\lambda^*(x) \coloneqq \inf \left\{ s > 0 : \sum_{n \in \mathbb{N}} \frac{1}{(d_n(x))^s} \text{ converges} \right\}
\end{align}
for each $x \in [0,1]$. Here, $(d_n(x))_{n \in \mathbb{N}}$ is the sequence of Pierce expansion {\em digits} of $x$. The Pierce expansion and its digit sequence are defined as follows. Given $x \in [0,1]$, we first define
\begin{align} \label{Pierce algorithm 1}
d_1(x) \coloneqq 
\begin{cases}
\left\lfloor 1/x \right\rfloor, &\text{if } x \neq 0; \\
\infty, &\text{if } x = 0,
\end{cases}
\quad \text{and} \quad 
T(x) \coloneqq 
\begin{cases}
1 - d_1 (x) x, &\text{if } x \neq 0; \\
0, &\text{if } x = 0,
\end{cases}
\end{align}
where $\lfloor y \rfloor$ denotes the largest integer not exceeding $y \in \mathbb{R}$. Now, for each $n \in \mathbb{N}$, we inductively define 
\begin{align} \label{Pierce algorithm 2}
d_n(x) \coloneqq d_1 (T^{n-1}(x))
\end{align}
to obtain the Pierce expansion digit sequence of $x$. It is well-known that every $x \in [0,1]$ admits its Pierce expansion of the form
\begin{align} \label{Pierce expansion} 
\begin{aligned}
x 
&= \sum_{k=1}^\infty \left( (-1)^{k+1} \prod_{n=1}^k \frac{1}{d_n(x)} \right) \\
&= \frac{1}{d_1(x)} - \frac{1}{d_1(x)d_2(x)} + \dotsb + \frac{(-1)^{k+1}}{d_1(x) \dotsm d_k(x)} + \dotsb
\end{aligned}
\end{align}
(see, e.g., \cite{Ahn23a, Pie29, Sch95, Sha86}).

Due to Lemma \ref{basic results on lambda star}(\ref{basic results on lambda star 2}) below, $\lambda^*$ is $[0,1]$-valued, and so, we may write $\lambda^* \colon [0,1] \to [0,1]$ in our results. Regarding the basic properties of $\lambda^*$, the main results of this paper are as follows:

\begin{theorem} \label{lambda star is zero a.e.}
The convergence exponent $\lambda^* \colon [0,1] \to [0,1]$ vanishes Lebesgue-almost everywhere.
\end{theorem}

\begin{theorem} \label{lambda star is surjective}
Any non-empty open subset $U$ of $[0,1]$ satisfies $\lambda^*(U \setminus \mathbb{Q}) = [0,1]$, and, consequently, $\lambda^*(U) = [0,1]$. In particular, $\lambda^* \colon [0,1] \to [0,1]$ is surjective.
\end{theorem}

\begin{corollary} \label{lambda star has intermediate value property}
The convergence exponent $\lambda^* \colon [0,1] \to [0,1]$ has the intermediate value property, i.e., for any $x, y \in [0,1]$ with $x<y$ and $\lambda^*(x) \neq \lambda^*(y)$, if $c$ is strictly between $\lambda^*(x)$ and $\lambda^*(y)$, then there exists $z \in (x,y)$ such that $\lambda^*(z) = c$.
\end{corollary}

\begin{corollary} \label{lambda star discontinuous corollary}
The convergence exponent $\lambda^* \colon [0,1] \to [0,1]$ is discontinuous everywhere.
\end{corollary}

Recall that a function $g \colon X \to \mathbb{R}$ on a metrizable space $X$ is said to be {\em of Baire class one (on $X$)} if there exists a sequence of real-valued continuous functions $(g_n)_{n \in \mathbb{N}}$ such that $g_n \to g$ pointwise as $n \to \infty$ everywhere on $X$. Similarly, we say that $g \colon X \to \mathbb{R}$ is {\em of Baire class two (on $X$)} if $g$ is the pointwise limit of some sequence of functions $(g_n)_{n \in \mathbb{N}}$, where each $g_n$ is of Baire class one on $X$. (See \cite[Chapter 24]{Kec95} for basic definitions and results on the Baire classification of functions.) 

\begin{corollary} \label{lambda star not first Baire corollary}
The convergence exponent $\lambda^* \colon [0,1] \to [0,1]$ is not of Baire class one.
\end{corollary}

\begin{theorem} \label{lambda star second Baire}
The convergence exponent $\lambda^* \colon [0,1] \to [0,1]$ is of Baire class two.
\end{theorem}

For each $\alpha \in [0,1]$, we denote by $L_\alpha^*$ the $\alpha$-level set of $\lambda^* \colon [0,1] \to [0,1]$, i.e.,
\[
L_\alpha^* \coloneqq \{ x \in [0,1] : \lambda^* (x) = \alpha \}.
\] 
By Theorem \ref{lambda star is zero a.e.}, we see that $L_0^*$ is of positive Lebesgue measure, hence of full Hausdorff dimension (see \cite[p.~45]{Fal14}), while $L_\alpha^*$ is of null Lebesgue measure for every $\alpha \in (0,1]$. We shall determine dimensions and topological sizes of $L_\alpha^*$ for each $\alpha \in [0,1]$. (See Section \ref{Preliminaries} for the definitions of the Hausdorff dimension and the box-counting dimension. We refer the reader to \cite{Fal14} for further definitions and results on dimensions.) Specifically, our main results in this regard are as follows:

\begin{theorem} \label{hdim L alpha star theorem}
The $\alpha$-level set $L_\alpha^*$ has Hausdorff dimension $1-\alpha$ for each $\alpha \in [0,1]$.
\end{theorem}

As is well known, a subset $A$ of a topological space $X$ is said to be {\em meager} in $X$ if $A$ can be written as a countable union of nowhere dense sets in $X$; otherwise, $A$ is {\em nonmeager} in $X$. If $X \setminus A$ is meager in $X$, we say that $A$ is {\em comeager} in $X$. (We refer the reader to \cite{Oxt80} for further details on the Baire category of sets.)

\begin{theorem} \label{L alpha star dense G delta theorem}
The $\alpha$-level set $L_\alpha^*$ is dense but not $G_\delta$ in $[0,1]$ for each $\alpha \in [0,1)$, and the $1$-level set $L_1^*$ is dense $G_\delta$ in $[0,1]$. Consequently, $L_\alpha^*$ is meager in $[0,1]$ for each $\alpha \in [0,1)$, and $L_1^*$ is comeager and nonmeager in $[0,1]$.
\end{theorem}

\begin{corollary} \label{bdim L alpha star corollary}
The $\alpha$-level set $L_\alpha^*$ has box-counting dimension $1$ for each $\alpha \in [0,1]$.
\end{corollary}

\begin{theorem} \label{L alpha star is uncountable}
The $\alpha$-level set $L_\alpha^*$ has cardinality $\mathfrak{c}$, i.e., the cardinality of $\mathbb{R}$, for each $\alpha \in [0,1]$.
\end{theorem}

Lastly, we shall examine the subset $\mathcal{D}^{(s)}_{\Div}$ of $[0,1]$ concerning the series of positive $s$th powers of the reciprocals of the digits, which is defined as
\begin{align} \label{definition of Theta}
\mathcal{D}^{(s)}_{\Div} \coloneqq \left\{ x \in [0,1] : \sum_{n \in \mathbb{N}} \frac{1}{(d_n(x))^{s}} \text{ diverges} \right\}
\end{align}
for each $s>0$. Note that $\mathcal{D}^{(s)}_{\Div}$ implicitly emerges from the definition \eqref{definition of lambda star} of the convergence exponent. In fact, we have
\begin{align} \label{equivalent definition of lambda star}
\lambda^*(x) = \inf \{ s > 0 : x \not \in \mathcal{D}^{(s)}_{\Div} \}
\end{align}
for each $x \in [0,1]$. Before stating the main results, we note that our results significantly generalize the theorems established by Shallit \cite{Sha86} described in Corollary \ref{Shallit's theorem} below, concerning some exceptional sets in Pierce expansions.

Now, we are in a position to list the main results on $\mathcal{D}^{(s)}_{\Div}$. Observe first that $\mathcal{D}^{(s)}_{\Div} = \varnothing$ if $s > 1$, by \eqref{equivalent definition of lambda star} and Theorem \ref{lambda star is surjective}.

\begin{theorem} \label{Theta Lebesgue measure theorem}
The set $\mathcal{D}^{(s)}_{\Div}$ has null Lebesgue measure for each $s \in (0,1]$.
\end{theorem}

Due to the preceding theorem, the set $\mathcal{D}^{(s)}_{\Div}$, for each $s \in (0,1]$, is an exceptional set (in the Lebesgue measure sense) arising in Pierce expansions. We determine the dimensions and topological sizes of these exceptional sets, just as we did for the level sets $L_\alpha^*$ above.

\begin{theorem} \label{Theta hdim theorem}
The set $\mathcal{D}^{(s)}_{\Div}$ has Hausdorff dimension $1-s$ for each $s \in (0,1]$.
\end{theorem}

\begin{theorem} \label{Theta dense G delta theorem}
The set $\mathcal{D}^{(s)}_{\Div}$ is dense $G_\delta$, hence comeager and nonmeager, in $[0,1]$, and has cardinality $\mathfrak{c}$, for each $s \in (0,1]$.
\end{theorem}

\begin{corollary} \label{bdim Theta corollary}
The set $\mathcal{D}^{(s)}_{\Div}$ has box-counting dimension $1$ for each $s \in (0,1]$.
\end{corollary}

Now, put
\begin{align} \label{definition of E and F}
\begin{aligned}
\mathcal{D}_{\Div} &\coloneqq \left\{ x \in [0,1] : \sum_{n \in \mathbb{N}} \frac{1}{d_n(x)} \text{ diverges} \right\}, \\
\mathcal{T}_{\Div} &\coloneqq \left\{ x \in [0,1] : \sum_{n \in \mathbb{N}} T^n(x) \text{ diverges} \right\},
\end{aligned}
\end{align} 
both of which were first studied in \cite{Sha86}. It should be pointed out that the sets $\mathcal{D}_{\Div}$ and $\mathcal{T}_{\Div}$ are, in fact, equal (see Lemma \ref{E and F are equal} below), although in \cite{Sha86}, this fact was not explicitly stated and the two sets were treated separately.

\begin{corollary} \label{sum of reciprocals hdim corollary}
The sets $\mathcal{D}_{\Div}$ and $\mathcal{T}_{\Div}$ are of null Hausdorff dimension.
\end{corollary}

\begin{corollary} \label{sum of reciprocals dense G delta corollary}
The sets $\mathcal{D}_{\Div}$ and $\mathcal{T}_{\Div}$ are dense $G_\delta$, hence comeager and nonmeager, in $[0,1]$.
\end{corollary}

\begin{corollary} [{\cite[Theorems 12 and 13]{Sha86}}] \label{Shallit's theorem}
The sets $\mathcal{D}_{\Div}$ and $\mathcal{T}_{\Div}$ have null Lebesgue measure, cardinality $\mathfrak{c}$, and are dense in $[0,1]$.
\end{corollary}

This paper is organized as follows. In Section \ref{Preliminaries}, we present preliminary facts about Pierce expansions, Hausdorff and box-counting dimensions, the Baire classification of functions, the Baire category of sets, and the cardinality of sets. In Section \ref{Auxiliary results}, we establish some auxiliary results that will be used in the proofs of the main results. In Section \ref{Proofs of main results}, we prove the aforementioned main results of this paper.

Throughout the paper, the closed unit interval $[0,1]$ is equipped with the subspace topology induced by the standard topology on $\mathbb{R}$. For a subset $F$ of $[0,1]$, we denote by $\diam F$ the diameter, $\hdim F$ the Hausdorff dimension, $\bdim F$ the box-counting dimension (if exists), and $\mathcal{H}^s(F)$ the $s$-dimensional Hausdorff measure. We denote by $\mathbb{N}$ the set of positive integers, $\mathbb{N}_0$ the set of non-negative integers, and $\mathbb{N}_\infty \coloneqq \mathbb{N} \cup \{ \infty \}$ the set of extended positive integers. The set of irrationals in $[0,1]$ will be denoted by $\mathbb{I}$. Following the convention, we define $c + \infty \coloneqq \infty$ and $c/\infty \coloneqq 0$ for any $c \in \mathbb{R}$, and $\infty^s \coloneqq \infty$ for any $s>0$. For any set $A$, we denote by $|A|$ its cardinality.

\section{Preliminaries} \label{Preliminaries}

In this section, we present preliminary facts about Pierce expansions, Hausdorff and box-counting dimensions, the Baire classification of functions, the Baire category of sets, and the cardinality of sets.

Since the convergence exponent was originally defined (see \cite[pp.~25--26]{PS98}) and extensively studied (see, e.g., \cite{Kos79, KS82, Sal84}) as a function on a certain sequence space, to investigate $\lambda^* \colon [0,1] \to [0, \infty]$, it is natural to consider the convergence exponent defined on a ``nice'' sequence space, which is inherently, in some sense, related to the Pierce expansion digit sequences. In \cite{Ahn23a}, where the error-sum function of Pierce expansions was discussed, we introduced the space of {\em Pierce sequences} $\Sigma$ in $\mathbb{N}_\infty^{\mathbb{N}}$ which is defined as
\[
\Sigma \coloneqq \Sigma_0 \cup \bigcup_{n \in \mathbb{N}} \Sigma_n \cup \Sigma_\infty,
\]
where
\begin{align*}
\Sigma_0 &\coloneqq \{ (\sigma_k)_{k \in \mathbb{N}} \in \{ \infty \}^{\mathbb{N}} \} = \{ (\infty, \infty, \dotsc) \}, \\
\Sigma_n &\coloneqq \{ (\sigma_k)_{k \in \mathbb{N}} \in \mathbb{N}^{\{ 1, \dotsc, n \}} \times \{ \infty \}^{\mathbb{N} \setminus \{ 1, \dotsc, n \}} : \sigma_1 < \sigma_2 < \dotsb < \sigma_n \}, \quad n \in \mathbb{N}, \\
\Sigma_\infty &\coloneqq \{ (\sigma_k)_{k \in \mathbb{N}} \in \mathbb{N}^{\mathbb{N}} : \sigma_k < \sigma_{k+1} \text{ for all } k \in \mathbb{N} \}.
\end{align*}

We equip $\mathbb{N}$ with the discrete topology and consider $\mathbb{N}_\infty$ as its one-point compactification. Then, the product space $\mathbb{N}_\infty^{\mathbb{N}}$ is compact by the Tychonoff theorem. Moreover, since $\mathbb{N}_\infty^{\mathbb{N}}$ is metrizable (\cite[Lemmas 3.1 and 3.3]{Ahn23a}), we may fix a compatible metric $\rho$ which induces the product topology on $\mathbb{N}_\infty^{\mathbb{N}}$. Henceforth, we write $\mathbb{N}_\infty^{\mathbb{N}}$ for both the product space and the metric space $(\mathbb{N}_\infty^{\mathbb{N}}, \rho)$. Similarly, we write $\Sigma$ for the subspace of the product space $\mathbb{N}_\infty^{\mathbb{N}}$ and also for the metric space $(\Sigma, \rho)$.

\begin{proposition} [{\cite[Lemmas 3.1--3.4]{Ahn23a}}] \label{Sigma is compact metric}
The space of Pierce sequences $\Sigma$ is a closed, hence compact, metric space as a subspace of the compact metric space $\mathbb{N}_\infty^{\mathbb{N}}$. Consequently, $\Sigma$ is a complete metric space.
\end{proposition}

The following proposition is not explicitly used in the subsequent discussions, but it is worth noting.

\begin{proposition} \label{Sigma infty is dense G delta proposition}
The subspace $\Sigma_\infty$ is dense $G_\delta$ in $\Sigma$.
\end{proposition}

\begin{proof}
Let $\sigma \coloneqq (\sigma_k)_{k \in \mathbb{N}} \in \Sigma \setminus \Sigma_\infty$. Then, $\sigma \in \Sigma_n$ for some $n \in \mathbb{N}_0$; moreover, $\sigma = (\infty, \infty, \dotsc)$ if $n =0$, and $\sigma = (\sigma_1, \dotsc, \sigma_n, \infty, \infty, \dotsc)$ with $\sigma_n \neq \infty$ if $n \geq 1$. For each $j \in \mathbb{N}$, define 
\[
\bm{\tau}_j \coloneqq 
\begin{cases}
(j, j+1, j+2, \dotsc), &\text{if } n = 0; \\
(\sigma_1, \dotsc, \sigma_n, \sigma_n + j, \sigma_n + (j+1), \sigma_n + (j+2), \dotsc), &\text{if } n \geq 1.
\end{cases}
\]
Then, $(\bm{\tau}_j)_{j \in \mathbb{N}}$ is a sequence in $\Sigma_\infty$ converging to $\sigma$ as $j \to \infty$. This proves that $\Sigma_\infty$ is dense in $\Sigma$.

For the $G_\delta$-ness of $\Sigma_\infty$, we consider the complement $\Sigma \setminus \Sigma_\infty = \Sigma_0 \cup \bigcup_{n \in \mathbb{N}} \Sigma_n$. Here, for each $n \in \mathbb{N}$, $\Sigma_n$ is countable as a subset of the countable set $\mathbb{N}^{\{ 1, \dotsc, n \}} \times \{ \infty \}^{\mathbb{N} \setminus \{ 1, \dotsc, n \}}$. Since $\Sigma_0 = \{ (\infty, \infty, \dotsc) \}$ is a singleton, we deduce that $\Sigma \setminus \Sigma_\infty$ is countable. Note that any one-point subset of $\Sigma$ is closed in $\Sigma$ since $\Sigma$ is a metric space (Proposition \ref{Sigma is compact metric}). It follows that $\Sigma \setminus \Sigma_\infty$ is $F_\sigma$ in $\Sigma$ as a countable union of singletons.
\end{proof}

We make use of two functions between $[0,1]$ and $\Sigma$. These functions were introduced and discussed in detail in \cite{Ahn23a}. Denote by $f \colon [0,1] \to \Sigma$ the function defined by $f(x) \coloneqq (d_k(x))_{k \in \mathbb{N}}$ for each $x \in [0,1]$. Conversely, define $\varphi \colon \Sigma \to [0,1]$ by
\[
\varphi (\sigma) \coloneqq \sum_{k=1}^\infty \left( (-1)^{k+1} \prod_{j=1}^k \frac{1}{\sigma_j} \right) = \frac{1}{\sigma_1} - \frac{1}{\sigma_1 \sigma_2} + \dotsb + \frac{(-1)^{k+1}}{\sigma_1 \dotsm \sigma_k} + \dotsb
\]
for each $\sigma \coloneqq (\sigma_j)_{j \in \mathbb{N}}$ in $\Sigma$.

According to the following proposition, the preimage of a singleton under $\varphi$ is either a singleton or a doubleton.

\begin{proposition}[See {\cite[Proposition 2.3]{Ahn23a}}] \label{preimage of phi}
Let $x \in [0,1]$. Then, the following hold:
\begin{enumerate}[label=\upshape(\roman*), ref=\roman*, leftmargin=*, widest=ii]
\item \label{preimage of phi 1}
If $x \in \mathbb{I} \cup \{ 0, 1 \}$, then $\varphi^{-1}( \{ x \}) = \{ f(x) \}$.
\item \label{preimage of phi 2}
If $x \in (0,1) \cap \mathbb{Q}$, then $\varphi^{-1}( \{ x \} ) = \{ \sigma, \tau \}$, where
\begin{align*}
\sigma &\coloneqq f(x) = (d_1(x), \dotsc, d_{n-1}(x), d_n(x), \infty, \infty, \dotsc) \in \Sigma_n, \\
\tau &\coloneqq (\underbrace{d_1(x), \dotsc, d_{n-1}(x)}_{\text{$n-1$ terms}}, d_n(x)-1, d_n(x), \infty, \infty, \dotsc) \in \Sigma_{n+1},
\end{align*}
for some $n \in \mathbb{N}$.
\end{enumerate}
\end{proposition}

\begin{proposition} [See {\cite[Subsection 3.2]{Ahn23a}}] \label{f is homeo}
For the mappings $f \colon [0,1] \to \Sigma$ and $\varphi \colon \Sigma \to [0,1]$, the following hold:
\begin{enumerate}[label=\upshape(\roman*), ref=\roman*, leftmargin=*, widest=iii]
\item \label{f is homeo 1}
$f$ is continuous at every point of $\mathbb{I} \cup \{ 0, 1 \}$ but is not continuous at any point of $\mathbb{Q} \cap (0,1)$.
\item \label{f is homeo 2}
$\varphi$ is continuous.
\item \label{f is homeo 3}
$f|_{\mathbb{I}} \colon \mathbb{I} \to \Sigma_\infty$, the restriction of $f$ to $\mathbb{I}$, is a homeomorphism with the continuous inverse $\varphi|_{\Sigma_\infty} : \Sigma_\infty \to \mathbb{I}$, the restriction of $\varphi$ to $\Sigma_\infty$.
\end{enumerate}
\end{proposition}

For each $\sigma \coloneqq (\sigma_k)_{k \in \mathbb{N}} \in \Sigma_n$, $n \in \mathbb{N}$, we define the {\em cylinder set} associated with $\sigma$ by
\begin{align*}
\Upsilon_\sigma
&\coloneqq \{ (\tau_k)_{k \in \mathbb{N}} \in \Sigma : \tau_k = \sigma_k \text{ for all } 1 \leq k \leq n \} \\
&= \Sigma \cap \left( \prod_{k=1}^n \{ \sigma_k \} \times \prod_{k > n} \mathbb{N}_\infty \right),
\end{align*}
which is clopen in $\Sigma$, and define the {\em fundamental interval} associated with $\sigma$ by
\begin{align} \label{definition of I sigma}
I_\sigma \coloneqq f^{-1}(\Upsilon_\sigma) = \{ x \in [0,1] : d_k(x) = \sigma_k \text{ for all } 1 \leq k \leq n \}.
\end{align}

\begin{proposition}[{\cite[Theorem 1]{Sha86}}] \label{I sigma}
Fix $\sigma \coloneqq (\sigma_k)_{k \in \mathbb{N}} \in \Sigma_n$ for some $n \in \mathbb{N}$. Put
\[
\widehat{\sigma} \coloneqq (\underbrace{\sigma_1, \dotsc, \sigma_{n-1}}_{\text{$n-1$ terms}}, \sigma_n+1, \infty, \infty, \dotsc) \in \Sigma_n.
\]
Then, $I_\sigma$ is the subinterval of $[0,1]$ with endpoints $\varphi (\sigma)$ and $\varphi (\widehat{\sigma})$. Consequently,
\begin{align} \label{diam I sigma}
\diam I_\sigma = | \varphi (\sigma) - \varphi (\widehat{\sigma})| = \left( \prod_{j=1}^{n} \frac{1}{\sigma_j} \right) \frac{1}{\sigma_n+1}.
\end{align}
\end{proposition}

We provide the definitions of the Hausdorff dimension and the box-counting dimension below.

\begin{definition} [See {\cite[Chapter 3]{Fal14}}] \label{hdim definition}
For a non-empty subset $F$ of $[0,1]$, the {\em Hausdorff dimension} of $F$, denoted $\hdim F$, is defined as
\[
\hdim F \coloneqq \inf \{ s \geq 0 : \mathcal{H}^s (F) = 0 \} = \sup \{ s : \mathcal{H}^s (F) = \infty \},
\]
where, for $s \geq 0$,
\[
\mathcal{H}^s (F) \coloneqq \lim_{\delta \to 0} \left( \inf \left\{ \sum_{k \in \mathbb{N}} (\diam U_k)^s : F \subseteq \bigcup_{k \in \mathbb{N}} U_k \text{ and } \diam U_k \in (0, \delta] \text{ for each } k \in \mathbb{N} \right\} \right).
\]
\end{definition}

\begin{definition} [See {\cite[Chapter 2]{Fal14}}] \label{bdim definition}
For a non-empty subset $F$ of $[0,1]$, the {\em box-counting dimension} of $F$, denoted $\bdim F$, is defined as
\[
\bdim F \coloneqq \lim_{\delta \to 0} \frac{\log N_\delta (F)}{-\log \delta},
\]
provided the limit exists, where $N_\delta(F)$ denotes the smallest number of sets of diameter not exceeding $\delta$ that cover $F$.
\end{definition}

We list two propositions which will be used in calculations of dimensions. Particularly, the first one provides a useful property of the box-counting dimension, while this property does not hold for the Hausdorff dimension in general.

\begin{proposition} [See {\cite[p.~37]{Fal14}}] \label{dense bdim proposition}
For any dense subset $F$ of $[0,1]$, we have $\bdim F = 1$.
\end{proposition}

\begin{proposition}[{\cite[Theorem 1.1]{Ahn23b}}] \label{hdim E alpha}
For each $\alpha \in (0,1]$, let
\[
E(\alpha) \coloneqq \left\{ x \in [0,1] : \lim_{n \to \infty} \frac{\log n}{\log d_n(x)} = \alpha \right\}.
\]
Then, $\hdim E(\alpha) = 1 - \alpha$ for each $\alpha \in (0,1]$.
\end{proposition}

\begin{remark}
Although the preceding proposition is sufficient for our subsequent discussions, it is worth noting that the result stated therein also holds when $\alpha=0$. In other words, we have
\[
\hdim E(0) = 1, \quad \text{where} \quad E(0) \coloneqq \left\{ x \in [0,1] : \lim_{n \to \infty} \frac{\log n}{\log d_n(x)} = 0 \right\}.
\]
Indeed, by the law of large numbers in Pierce expansions \cite[Theorem 16]{Sha86}, we know that the set
\[
\left\{ x \in [0,1] : \lim_{n \to \infty} \frac{n}{\log d_n(x)} = 1 \right\}
\]
has full Lebesgue measure on $[0,1]$. As this set is a subset of $E(0)$, we observe that $E(0)$ is of full Lebesgue measure on $[0,1]$, and consequently, $E(0)$ has full Hausdorff dimension.
\end{remark}

The last topic of this section covers some topological preliminaries. 

\begin{proposition} [See {\cite[Lemma 20.7]{Din74}}] \label{infimum of lsc is lsc proposition}
Let $X$ and $Y$ be compact metrizable spaces and $h \colon X \to Y$ a continuous surjection. Suppose that a function $g \colon X \to [-\infty, \infty]$ is lower semicontinuous. Then, the mapping
\[
y \mapsto \inf_{x \in h^{-1}(\{ y \})} g(x)
\]
is lower semicontinuous on $Y$.
\end{proposition}

We remark that in \cite{Din74}, the function $g$ was assumed to be non-negative real-valued. We include the proof to make it clear that such an assumption is not necessary in our proposition.

\begin{proof} [Proof of Proposition \ref{infimum of lsc is lsc proposition}]
It is not hard to verify that for each $y \in Y$, we have $\inf_{x \in h^{-1}(\{ y \})} g(x) = g(x')$ for some $x' \in h^{-1}( \{ y \})$. (One can find the details in \cite{Din74}.)

Let $\alpha \in \mathbb{R}$ be arbitrary. We show that the sublevel set
\[
A_\alpha \coloneqq \left\{ y \in Y : \inf_{x \in h^{-1}(\{ y \})} g(x) \leq \alpha \right\}
\]
is closed in $Y$. Let $(\eta_n)_{n \in \mathbb{N}}$ be a sequence in $A_\alpha$ converging to some $\eta \in Y$. As noted in the preceding paragraph, for each $n \in \mathbb{N}$, we can find $\xi_n \in X$ such that $\inf_{x \in h^{-1}(\{ \eta_n \})} g(x) = g(\xi_n)$ with $\xi_n \in h^{-1}(\{ \eta_n \})$. Since $X$ is compact, the sequence $(\xi_n)_{n \in \mathbb{N}}$ has a convergent subsequence $(\xi_{n_k})_{k \in \mathbb{N}}$ with limit, say $\xi$, in $X$. Now, by the continuity of $h$, we have $\eta_{n_k} = h(\xi_{n_k}) \to h(\xi)$ as $k \to \infty$. On the other hand, since $\eta_n \to \eta$ as $n \to \infty$ in the metrizable space $Y$, it follows that $h(\xi) = \eta$, or, equivalently, $\xi \in h^{-1}(\{ \eta \})$. Hence, $\inf_{x \in h^{-1}(\{ \eta \})} g(x) \leq g(\xi)$. Furthermore, since $g(\xi_n) \leq \alpha$ for each $n \in \mathbb{N}$ by the construction of $(\xi_n)_{n \in \mathbb{N}}$, and since $g$ is lower semicontinuous, we obtain $g(\xi) \leq \alpha$. Therefore, $\inf_{x \in h^{-1}(\{ \eta \})} g(x) \leq \alpha$, which implies $\eta \in A_\alpha$, completing the proof.
\end{proof}

A classical theorem by Baire on lower semicontinuous functions can be stated as follows:

\begin{proposition} [See {\cite[Theorem 23.19]{Kec95} and \cite[p.~132, Exercise 4(g)]{Str81}}] \label{lower semicontinuous function proposition}
Any $[-\infty,\infty]$-valued lower semicontinuous function on a metrizable space $X$ is the pointwise limit of some sequence of real-valued continuous functions on $X$.
\end{proposition}

\begin{proof}
Let $X$ be a metrizable space and $g \colon X \to [-\infty, \infty]$ a lower semicontinuous function. Put $Y \coloneqq \{ x \in X : g(x) = - \infty \}$, which is closed in $X$ by the lower semicontinuity of $g$. By \cite[p.~132, Exercise 4(g)]{Str81}, there is a sequence of real-valued continuous functions $(g_n)_{n \in \mathbb{N}}$ on $X \setminus Y$ such that $\lim_{n \to \infty} g_n (x) = g(x)$ for all $x \in X \setminus Y$. 

If $Y$ is empty, then we are done. So we assume that $Y$ is non-empty. Let $d$ be a compatible metric on $X$. Define $d (x, Y) \coloneqq \inf \{ d(x,y) : y \in Y \}$ for each $x \in X$. It is a standard fact that the mapping $x \mapsto d(x,Y)$ is continuous on $X$. For each $n \in \mathbb{N}$, let
\[
Z_n \coloneqq \left\{ x \in X : d (x, Y) \geq \frac{1}{n} \right\}.
\]
Then, $Z_n$ is closed in $X$, by the continuity of $d(x,Y)$, and $Z_n \cap Y = \varnothing$. Moreover, we note that $X = Y \cup \bigcup_{n \in \mathbb{N}} Z_n$. To see this, let, if possible, $x \in X \setminus \bigcup_{n \in \mathbb{N}} Z_n$. Then, $x \not \in Z_n$ for any $n \in \mathbb{N}$, and so, by definition, we have $d (x,Y) < 1/n$ for all $n \in \mathbb{N}$. Hence, $d (x,Y) = 0$, and this implies that $x \in Y$ since $Y$ is closed.

For each $n \in \mathbb{N}$, define
\[
\widetilde{g}_n(x) \coloneqq 
\begin{cases}
-n, &\text{if } x \in Y; \\
g_n(x), &\text{if } x \in Z_n,
\end{cases}
\]
for each $x \in Y \cup Z_n$. Then, $\widetilde{g}_n$ is a real-valued continuous function on the closed subset $Y \cup Z_n$ of $X$, for each $n \in \mathbb{N}$. By the Tietze extension theorem, for each $n \in \mathbb{N}$, there is a real-valued continuous extension of $\widetilde{g}_n$ to the whole space $X$, that is, there exists a continuous function $G_n : X \to \mathbb{R}$ satisfying $G_n (x) = \widetilde{g}_n (x)$ for all $x \in Y \cup Z_n$. Now, $(G_n)_{n \in \mathbb{N}}$ is a sequence of real-valued continuous functions on $X$; hence, it remains to show that $G_n \to g$ pointwise as $n \to \infty$ everywhere on $X$. If $x \in Y$, then $g(x) = -\infty = \lim_{n \to \infty} (-n) = \lim_{n \to \infty} \widetilde{g}_n(x) = \lim_{n \to \infty} G_n(x)$. On the other hand, assume $x \in X \setminus Y$. Then, $x \in Z_m$ for some $m \in \mathbb{N}$, as noted in the preceding paragraph. Then, since $Z_m \subseteq Z_{m+1} \subseteq \dotsb$, we have $G_n(x) = \widetilde{g}_n(x) = g_n(x)$ for all $n \geq m$. Hence, $g(x) = \lim_{n \to \infty} g_n(x) = \lim_{n \to \infty} G_n(x)$. This completes the proof of the proposition.
\end{proof}

The following poposition provides one property, regarding the Baire category and the Borel hierarchy, of the set of discontinuous points of a function of Baire class one.

\begin{proposition} [See {\cite[Theorem 24.14]{Kec95}}] \label{not first Baire proposition}
Let $g$ be a real-valued function on a metrizable space $X$. If $g$ is of Baire class one, then the set of points of discontinuity of $g$ is meager $F_\sigma$ in $X$.
\end{proposition}

Note that Proposition \ref{not first Baire proposition} in particular tells us that any nowhere continuous function on a complete metric space $X$ is not of Baire class one. In fact, since $X$ is nonmeager in itself by the Baire category theorem, $X$ cannot be the set of points of discontinuity of any function of Baire class one.

We state one consequence of the Cantor-Bendixson theorem on Polish spaces.

\begin{proposition} [See {\cite[Theorem 3.11 and Corollary 6.5]{Kec95}}] \label{cardinality of uncoutable Polish}
Let $X$ be a Polish space, i.e., $X$ is separable and completely metrizable. If $A$ is an uncountable $G_\delta$ subset of $X$, then $A$ has cardinality $\mathfrak{c}$.
\end{proposition}

The following topological facts are elementary and well-known. For the convenience of the readers, we provide proofs of these facts.

\begin{fact} \label{basic facts}
Suppose that $A$ is dense $G_\delta$ in a complete metric space $X$. Then, the following hold:
\begin{enumerate} [label=\upshape(\roman*), ref=\roman*, leftmargin=*, widest=ii]
\item \label{basic facts 1}
$A$ is comeager and nonmeager in $X$.
\item \label{basic facts 2}
If $X$ does not contain any isolated points, then $A$ is uncountable.
\end{enumerate}
\end{fact}

\begin{proof}
For both parts, since $A$ is dense $G_\delta$ in $X$, we may write $A = \bigcap_{n \in \mathbb{N}} U_n$, where each $U_n$ is dense open in $X$.

(\ref{basic facts 1})
Since each $X \setminus U_n$ is nowhere dense in $X$, we know that $X \setminus A = \bigcup_{n \in \mathbb{N}} (X \setminus U_n)$ is meager in $X$. Hence, $A$ is comeager in $X$. Now, if $A$ were meager in $X$, then $X = A \cup (X \setminus A)$ would be meager in $X$ as well, which contradicts the Baire category theorem stating that any complete metric space is nonmeager in itself. Thus, $A$ is nonmeager in $X$.

(\ref{basic facts 2})
Suppose that $X$ has no isolated points. Assume to the contrary that $A$ is countable, and put $A = \bigcup_{n \in \mathbb{N}} \{ a_n \}$. Since each $a_n$ is not an isolated point of $X$, we know that each $V_n \coloneqq X \setminus \{ a_n \}$ is dense open in $X$. Then, $\bigcap_{n \in \mathbb{N}} V_n \cap \bigcap_{n \in \mathbb{N}} U_n = (X \setminus A) \cap A = \varnothing$, which contradicts the Baire category theorem stating that in any complete metric space, a countable intersection of dense open sets is dense, and, in particular, non-empty.
\end{proof}

\section{Auxiliary results} \label{Auxiliary results}

We begin this section with some general results which are independent of Pierce expansions.

\begin{lemma} \label{at most one G delta level set}
Let $X$ be a complete metric space. Suppose that $X = \bigcup_{\alpha \in \Lambda} X_\alpha$ for some index set $\Lambda$, where $X_\alpha \cap X_\beta = \varnothing$ whenever $\alpha, \beta \in \Lambda$ with $\alpha \neq \beta$. Then, at most one element in the collection $\{ X_\alpha : \alpha \in\Lambda \}$ is dense $G_\delta$ in $X$.
\end{lemma}

\begin{proof}
Suppose to the contrary that there exist two distinct dense $G_\delta$ sets $X_\alpha$ and $X_\beta$ in the collection $\{ X_\alpha : \alpha \in \Lambda \}$. Write $X_\alpha = \bigcap_{n \in \mathbb{N}} U_n$ and $X_\beta = \bigcap_{n \in \mathbb{N}} V_n$, where $U_n$ and $V_n$ are dense open in $X$ for all $n \in \mathbb{N}$. Then, by the Baire category theorem, $X_\alpha \cap X_\beta = \bigcap_{n \in \mathbb{N}} (U_n \cap V_n)$, as a countable intersection of dense open sets, is dense in $X$, and, in particular, non-empty. This is a contradiction, and thus, the lemma is proved.
\end{proof}

\begin{lemma} \label{superlevel set of limsup of lsc is G delta}
Let $(g_n)_{n \in \mathbb{N}}$ be a sequence of $[-\infty,\infty]$-valued lower semicontinuous functions on a topological space $X$. Put $g(x) \coloneqq \limsup_{n \to \infty} g_n(x)$ for each $x \in X$. Then, the superlevel set $\{ x \in X : g(x) \geq \alpha \}$ is $G_\delta$ in $X$ for each $\alpha \in \mathbb{R}$.
\end{lemma}

\begin{proof}
Let $\alpha \in \mathbb{R}$. Write
\begin{align*}
\{ x \in X : g(x) \geq \alpha \}
&= \left\{ x \in X : \inf_{n \in \mathbb{N}} \left( \sup_{k \geq n} g_k(x) \right) \geq \alpha \right\} \\
&= \bigcap_{n \in \mathbb{N}} \left\{ x \in X : \sup_{k \geq n} g_k(x) \geq \alpha \right\} \\
&= \bigcap_{n \in \mathbb{N}} \bigcap_{j \in \mathbb{N}} \left\{ x \in X : \sup_{k \geq n} g_k(x) > \alpha - \frac{1}{j} \right\} \\
&= \bigcap_{n \in \mathbb{N}} \bigcap_{j \in \mathbb{N}} \bigcup_{k \geq n} \left\{ x \in X : g_k(x) > \alpha - \frac{1}{j} \right\}.
\end{align*}
Here, each $\{ x \in X : g_k(x) > \alpha-1/j \}$ is open in $X$ by the lower semicontinuity of $g_k$. Hence the lemma.
\end{proof}

\begin{lemma} \label{double limit Baire class two lemma}
Let $g \colon X \to \mathbb{R}$ be a function on a metrizable space $X$. Suppose that there exists a double sequence of $[-\infty,\infty]$-valued continuous functions $(g_{m,n})_{(m,n) \in \mathbb{N}^2}$ on $X$ such that $g(x) = \lim_{n \to \infty} \left( \lim_{m \to \infty} g_{m,n}(x) \right)$ for every $x \in X$. Then, $g$ is of Baire class two on $X$.
\end{lemma}

\begin{proof}
For each $(m, n) \in \mathbb{N}^2$ and $x \in X$, define
\[
\widetilde{g}_{m,n}(x) \coloneqq 
\begin{cases}
n, &\text{if } g_{m,n}(x) \in (n, \infty]; \\
g_{m,n}(x), &\text{if } g_{m,n}(x) \in [-n,n]; \\
-n, &\text{if } g_{m,n}(x) \in [-\infty, -n).
\end{cases}
\]
Clearly, $\widetilde{g}_{m,n}$ is real-valued and continuous on $X$ for each $(m,n) \in \mathbb{N}^2$. For each $n \in \mathbb{N}$ and $x \in X$, put $g_n(x) \coloneqq \lim_{m \to \infty} g_{m,n}(x)$. Given $n \in \mathbb{N}$, we have the pointwise convergence
\[
\widetilde{g}_{m,n}(x) \to \widetilde{g}_n(x) \coloneqq  
\begin{cases}
n, &\text{if } g_n(x) \in (n, \infty]; \\
g_n (x), &\text{if } g_n(x) \in [-n,n]; \\
-n, &\text{if } g_n(x) \in [-\infty, -n),
\end{cases}
\]
as $m \to \infty$ for any $x \in X$. Note that $\widetilde{g}_n$ is real-valued, and hence $\widetilde{g}_n$ is of Baire class one on $X$, for each $n \in \mathbb{N}$. Finally, since $g(x) = \lim_{n \to \infty} g_n(x) = \lim_{n \to \infty} \widetilde{g}_n(x)$ for every $x \in X$ and since $g$ is real-valued by the hypothesis, we conclude that $g$ is of Baire class two on $X$.
\end{proof}

\begin{lemma} \label{main lemma}
Let $X$ and $Y$ be compact metrizable spaces and $h \colon X \to Y$ a continuous surjection. Suppose that $(g_n)_{n \in \mathbb{N}}$ is a sequence of $[-\infty,\infty]$-valued lower semicontinuous functions on $X$. Then, the function $\Phi$ on $Y$ defined by
\[
\Phi (y) \coloneqq \inf_{x \in h^{-1}(\{ y \})} \left( \limsup_{n \to \infty} g_n (x) \right)
\]
for each $y \in Y$, is the double pointwise limit of some double sequence of real-valued continuous functions on $Y$. In particular, if $\Phi$ is real-valued, then $\Phi$ is of Baire class two on $Y$.
\end{lemma}

\begin{proof}
For each $m \in \mathbb{N}$, define a map $\Phi_m$ on $Y$ by
\[
\Phi_m (y) \coloneqq \inf_{x \in h^{-1}(\{y\})} G_m(x)
\]
for each $y \in Y$, where the map $G_m$ on $X$ is given by
\[
G_m (x) \coloneqq \sup_{n \geq m} g_n(x)
\]
for each $x \in X$. By noting that $(G_m)_{m \in \mathbb{N}}$ is a decreasing sequence of functions, we find that
\begin{align*}
\Phi (y)
&= \inf_{x \in h^{-1}(\{ y \})} \left( \lim_{m \to \infty} G_m(x) \right) 
= \inf_{x \in h^{-1}(\{ y \})} \left( \inf_{m \in \mathbb{N}} G_m (x) \right) \\
&= \inf_{m \in \mathbb{N}} \left( \inf_{x \in h^{-1}(\{y\})} G_m(x) \right) 
= \lim_{m \to \infty} \left( \inf_{x \in h^{-1}(\{y\})} G_m(x) \right)
= \lim_{m \to \infty} \Phi_m (y),
\end{align*}
for each $y \in Y$. Hence, it suffices to show that each $\Phi_m$ is the pointwise limit of some sequence of real-valued continuous functions on $Y$. Let $m \in \mathbb{N}$. By writing 
\[
\{ x \in X : G_m (x) > \alpha \} = \bigcup_{n \geq m} \{ x \in X : g_n(x) > \alpha \}
\]
for any $\alpha \in \mathbb{R}$, we deduce that the map $G_m \colon X \to [-\infty, \infty]$ is lower semicontinuous, since each $\{ x \in X : g_n(x) > \alpha \}$ is open in $X$ by the lower semicontinuity of $g_n$. But then, $\Phi_m \colon Y \to [-\infty, \infty]$ is lower semicontinuous by Proposition \ref{infimum of lsc is lsc proposition}. Therefore, Baire's theorem on lower semicontinuous functions (Proposition \ref{lower semicontinuous function proposition}) tells us that $\Phi_m$ is the pointwise limit of some sequence of real-valued continuous functions on $Y$. This proves the first part of the lemma.

Now, assume that $\Phi$ is real-valued. By the preceding paragraph, we can find a double sequence of real-valued continuous functions $(\phi_{m,n})_{(m,n) \in \mathbb{N}^2}$ on $Y$ such that $\Phi (y) = \lim_{n \to \infty} \left( \lim_{m \to \infty} \phi_{m,n}(y) \right)$ for each $y \in Y$. Consequently, Lemma \ref{double limit Baire class two lemma} tells us that $\Phi$ is of Baire class two on $Y$.
\end{proof}

Now, we examine the convergence exponent of Pierce sequences $\lambda \colon \Sigma \to [0,\infty]$ which is defined as
\begin{align} \label{definition of lambda}
\lambda (\sigma) \coloneqq \inf \left\{ s > 0 : \sum_{n \in \mathbb{N}} \frac{1}{\sigma_n^s} \text{ converges} \right\}
\end{align}
for each $\sigma \coloneqq (\sigma_n)_{n \in \mathbb{N}}$ in $\Sigma$. Then, in view of \eqref{definition of lambda star}, we have that
\begin{align} \label{lambda star as composition}
(\lambda \circ f) (x) = \lambda ((d_n(x))_{n \in \mathbb{N}}) = \lambda^*(x)
\end{align}
for each $x \in [0,1]$, i.e., the following diagram commutes:
\begin{center}
\begin{tikzcd}[column sep=small]
{[0,1]} \arrow{rr}{\lambda^*} \arrow[swap, shift right=.75ex]{dr}{f}& &{[0, \infty]} \\
& \Sigma \arrow[swap]{ur}{\lambda}  
\end{tikzcd}
\end{center}

For each $n \in \mathbb{N} \setminus \{ 1 \}$, define a function $\psi_n \colon \Sigma \to [0,1]$ by
\begin{align} \label{definition of psi n}
\psi_n (\sigma) \coloneqq \frac{\log n}{\log {\sigma_n}}
\end{align}
for each $\sigma \coloneqq (\sigma_k)_{k \in \mathbb{N}}$ in $\Sigma$. By the convention $c / \infty = 0$ for any $c \in \mathbb{R}$ and the fact that $\sigma_k > 1$ for all $k \in \mathbb{N} \setminus \{ 1 \}$ for any $\sigma \coloneqq (\sigma_k)_{k \in \mathbb{N}} \in \Sigma$, it is clear that $\psi_n$ is well-defined. Furthermore, the continuity of $\psi_n$ is immediate from the definition.

In {\cite[pp.~25--26]{PS98}}, the authors provided a simple formula for the convergence exponent as follows. If $\sigma \coloneqq (\sigma_n)_{n \in \mathbb{N}}$ is a sequence of non-decreasing positive real numbers such that $\sigma_n \to \infty$ as $n \to \infty$, then $\lambda (\sigma) = \limsup_{n \to \infty} (\log n/\log \sigma_n)$. In fact, the same formula is still valid on $\Sigma$.

\begin{lemma} \label{PS formula on Sigma lemma}
For any $\sigma \in \Sigma$, we have
\begin{align} \label{PS formula on Sigma}
\lambda (\sigma) = \limsup_{n \to \infty} \psi_n(\sigma).
\end{align}
\end{lemma}

\begin{proof}
Note first that any $\sigma \coloneqq (\sigma_n)_{n \in \mathbb{N}}$ in $\Sigma_\infty$ is a non-decreasing positive real-valued sequence satisfying $\sigma_n \to \infty$ as $n \to \infty$. Hence, by {\cite[pp.~25--26]{PS98}} stated above the lemma, the formula \eqref{PS formula on Sigma} holds on $\Sigma_\infty$. 

Now, let $\sigma \coloneqq (\sigma_k)_{k \in \mathbb{N}} \in \Sigma_M$ for some $M \in \mathbb{N}_0$. Then, by definition, $\sigma_k = \infty$ for all $k \geq M+1$. It follows that $\sum_{k \in \mathbb{N}} (1/\sigma_k^s) = \sum_{k=1}^{M} (1/\sigma_k^s)$ is finite for any $s > 0$, which implies $\lambda (\sigma) = 0$ by the definition \eqref{definition of lambda}. On the other hand, it is clear that $\limsup_{k \to \infty} \psi_k (\sigma) = \limsup_{k \to \infty} (\log k / \log \sigma_k) = 0$ due to the convention $c/\infty = 0$ for any $c \in \mathbb{R}$. This completes the proof.
\end{proof}

\begin{lemma} \label{basic results on lambda}
For $\lambda \colon \Sigma \to [0, \infty]$, the following hold:
\begin{enumerate}[label=\upshape(\roman*), ref=\roman*, leftmargin=*, widest=ii]
\item \label{basic results on lambda 1}
For any $\sigma \in \bigcup_{n \in \mathbb{N}_0} \Sigma_n$, we have $\lambda (\sigma) = 0$.
\item \label{basic results on lambda 2}
$\lambda (\Sigma) \subseteq [0,1]$.
\end{enumerate}
\end{lemma}

\begin{proof}
The first part is already shown in the proof of Lemma \ref{PS formula on Sigma lemma}. For the second part, recall that for any $(\sigma_n)_{n \in \mathbb{N}} \in \Sigma$, we have $\sigma_n \geq n$ for all $n \in \mathbb{N}$. Hence, the result follows from \eqref{PS formula on Sigma}.
\end{proof}

Due to Lemma \ref{basic results on lambda}(\ref{basic results on lambda 2}), we may consider $\lambda$ as a $[0,1]$-valued function.

\begin{theorem} \label{neighborhood image theorem}
For any non-empty open subset $U$ of $\Sigma$, we have $\lambda (U \cap \Sigma_\infty) = [0,1]$, and, consequently, $\lambda (U) = [0,1]$. In particular, $\lambda \colon \Sigma \to [0,1]$ is surjective.
\end{theorem}

\begin{proof}
Let $U$ be a non-empty open subset of $\Sigma$. By definitions of the product topology and the subspace topology, we can find a cylinder set $\Upsilon_\sigma$ contained in $U$ for some $\sigma \in \Sigma_M$ with $M \in \mathbb{N}$. Put $\sigma \coloneqq (\sigma_j)_{j \in \mathbb{N}} = (\sigma_1, \dotsc, \sigma_M, \infty, \infty, \dotsc)$. For each $\alpha \in [0,1]$, define
\[
\sigma^{(\alpha)} \coloneqq 
\begin{cases}
( \sigma_1, \dotsc, \sigma_M, (\sigma_M+1)^{M+1}, (\sigma_M+2)^{M+2}, (\sigma_M+3)^{M+3}, \dotsc), &\text{if } \alpha = 0; \\
( \sigma_1, \dotsc, \sigma_M, \lfloor (\sigma_M+1)^{1/\alpha} \rfloor, \lfloor (\sigma_M+2)^{1/\alpha} \rfloor, \lfloor (\sigma_M+3)^{1/\alpha} \rfloor, \dotsc), &\text{if } \alpha \in (0,1].
\end{cases}
\]
Then, $\sigma^{(\alpha)} \in \Upsilon_\sigma \cap \Sigma_\infty \subseteq U \cap \Sigma_\infty$ for each $\alpha \in [0,1]$. By using \eqref{PS formula on Sigma}, we find that
\[
\lambda (\sigma^{(\alpha)}) 
= \begin{dcases}
\limsup_{k \to \infty} \frac{\log k}{\log (\sigma_M+k-M)^{k}} = 0, &\text{if } \alpha = 0; \\
\limsup_{k \to \infty} \frac{\log k}{\log \lfloor (\sigma_M+k-M)^{1/\alpha} \rfloor} = \alpha, &\text{if } \alpha \in (0,1],
\end{dcases}
\]
and so $\lambda (\sigma^{(\alpha)}) = \alpha$ for each $\alpha \in [0,1]$. Thus,
\[
[0,1] = \bigcup_{\alpha \in [0,1]} \lambda (\{ \sigma^{(\alpha)} \}) = \lambda \left( \bigcup_{\alpha \in [0,1]} \{ \sigma^{(\alpha)} \} \right) \subseteq \lambda (U \cap \Sigma_\infty) \subseteq \lambda (U) \subseteq [0,1],
\]
where the last inclusion follows from Lemma \ref{basic results on lambda}(\ref{basic results on lambda 2}), and this establishes the theorem.
\end{proof}

\begin{corollary} \label{lambda discontinuous corollary}
The convergence exponent $\lambda : \Sigma \to [0, 1]$ is discontinuous everywhere.
\end{corollary}

\begin{proof}
The corollary is an immediate consequence of Theorem \ref{neighborhood image theorem}.
\end{proof}

\begin{corollary} \label{lambda not first Baire corollary}
The convergence exponent $\lambda \colon \Sigma \to [0,1]$ is not of Baire class one.
\end{corollary}

\begin{proof}
Since $\lambda \colon \Sigma \to [0,1]$ is discontinuous everywhere (Corollary \ref{lambda discontinuous corollary}) on the complete metric space $\Sigma$ (Proposition \ref{Sigma is compact metric}), the corollary follows from Proposition \ref{not first Baire proposition}.
\end{proof}

However, it turns out that $\lambda \colon \Sigma \to [0,1]$ is of Baire class two.

\begin{theorem} \label{lambda second Baire}
The convergence exponent $\lambda \colon \Sigma \to [0,1]$ is of Baire class two.
\end{theorem}

\begin{proof}
For each $\sigma \in \Sigma$, the formula \eqref{PS formula on Sigma} yields $\lambda (\sigma) = \lim_{n \to \infty} \left( \lim_{m \to \infty} \lambda_{m,n} (\sigma) \right)$, where
\[
\lambda_{m,n}(\sigma) \coloneqq \begin{cases} \max_{n \leq k \leq m}  \psi_k (\sigma), &\text{if } n \leq m; \\ 0, &\text{if } n > m. \end{cases}
\]
Note that for each $(n,m) \in \mathbb{N}^2$ with $n \leq m$, the function $\psi_k \colon \Sigma \to [0,1]$ is continuous for each $k \in \{ n, \dotsc, m \}$ by definition, and so, $\lambda_{m,n}$ is continuous on $\Sigma$. Therefore, since $\lambda$ is real-valued by Lemma \ref{basic results on lambda}(\ref{basic results on lambda 2}), we conclude by Lemma \ref{double limit Baire class two lemma} that $\lambda$ is of Baire class two on $\Sigma$.
\end{proof}

For each $\alpha \in [0,1]$, we denote by $L_\alpha$  the $\alpha$-level set of $\lambda \colon \Sigma \to [0,1]$, i.e.,
\[
L_\alpha \coloneqq \{ \sigma \in \Sigma : \lambda (\sigma) = \alpha \}.
\]

\begin{theorem} \label{L alpha is dense in Sigma}
The $\alpha$-level set $L_\alpha$ is dense in $\Sigma$ for each $\alpha \in [0,1]$.
\end{theorem}

\begin{proof}
The theorem is clear in view of Theorem \ref{neighborhood image theorem}.
\end{proof}

For each $\alpha \in \mathbb{R}$, define the superlevel set and the sublevel set of $\lambda \colon \Sigma \to [0,1]$ by
\[
L_\alpha^+ \coloneqq \{ \sigma \in \Sigma : \lambda (\sigma) \geq \alpha \} \quad \text{and} \quad
L_\alpha^- \coloneqq \{ \sigma \in \Sigma : \lambda (\sigma) \leq \alpha \},
\]
respectively.

\begin{lemma} \label{L alpha + is G delta}
The superlevel set $L_\alpha^+$ is $G_{\delta}$ in $\Sigma$ for each $\alpha \in \mathbb{R}$. In particular, $L_\alpha^+$ is $G_\delta$ in $\Sigma_\infty$ for each $\alpha \in (0,1]$.
\end{lemma}

\begin{proof}
Since each $\psi_n \colon \Sigma \to [0,1]$ is continuous, hence lower semicontinuous, and since $\lambda (\sigma) = \limsup_{n \to \infty} \psi_n(\sigma)$ for each $\sigma \coloneqq (\sigma_n)_{n \in \mathbb{N}} \in \Sigma$ by \eqref{PS formula on Sigma}, it follows from Lemma \ref{superlevel set of limsup of lsc is G delta} that each $L_\alpha^+$ is $G_\delta$ in $\Sigma$ for each $\alpha \in \mathbb{R}$.

For the second assertion, it is enough to note that $L_\alpha^+ \subseteq \Sigma_\infty$ for each $\alpha \in (0,1]$ by Lemma \ref{basic results on lambda}(\ref{basic results on lambda 1}).
\end{proof}

\begin{lemma} \label{L1 is G delta}
The $1$-level set $L_1$ is $G_\delta$ in $\Sigma_\infty$, and in $\Sigma$.
\end{lemma}

\begin{proof}
Since $\lambda(\Sigma) \subseteq [0,1]$ by Lemma \ref{basic results on lambda}(\ref{basic results on lambda 2}), we have $L_1^+ = L_1$. Thus, Lemma \ref{L alpha + is G delta} tells us that $L_1$ is $G_\delta$ in $\Sigma_\infty$, and in $\Sigma$.
\end{proof}

\begin{theorem} \label{L alpha dense G delta theorem}
The $\alpha$-level set $L_\alpha$ is dense but not $G_\delta$ in $\Sigma$ for each $\alpha \in [0,1)$, and the $1$-level set $L_1$ is dense $G_\delta$ in $\Sigma$. Consequently, $L_\alpha$ is meager in $\Sigma$ for each $\alpha \in [0,1)$, and $L_1$ is comeager and nonmeager in $\Sigma$.
\end{theorem}

\begin{proof}
By Theorem \ref{L alpha is dense in Sigma}, we know that $L_\alpha$ is dense in $\Sigma$ for each $\alpha \in [0,1]$. We further know that $L_1$ is $G_\delta$ in $\Sigma$ from Lemma \ref{L1 is G delta}. So, it remains to show that $L_\alpha$ is not $G_\delta$ in $\Sigma$ for each $\alpha \in [0,1)$.

To this end, recall that $\Sigma$ is a complete metric space (Proposition \ref{Sigma is compact metric}). We can write $\Sigma$ as a disjoint union of the collection $\{ L_\alpha : \alpha \in [0,1] \}$. Then, we deduce from Lemma \ref{at most one G delta level set}, in conjunction with the preceding paragraph, that $L_1$ is the unique dense $G_\delta$ element in the collection $\{ L_\alpha : \alpha \in [0,1] \}$. But, for each $\alpha \in [0,1)$, $L_\alpha$ is dense in $\Sigma$ by Theorem \ref{L alpha is dense in Sigma}, and thus, $L_\alpha$ cannot be $G_\delta$ in $\Sigma$.

Now, we deduce that $L_1$ is comeager and nonmeager in $\Sigma$ in view of Fact \ref{basic facts}(\ref{basic facts 1}). It then follows that each $L_\alpha$, $\alpha \in [0,1)$, is meager in $\Sigma$ as a subset of the meager set $\Sigma \setminus L_1$.
\end{proof}

\begin{lemma} \label{cardinality of L alpha lemma}
We have $|L_\alpha \cap \Sigma_\infty| = \mathfrak{c}$ for each $\alpha \in [0,1]$.
\end{lemma}

\begin{proof}
Fix $\alpha \in [0,1]$. By the Cantor-Schr{\"o}der-Bernstein theorem, it suffices to prove two inequalities $|L_\alpha \cap \Sigma_\infty| \leq \mathfrak{c}$ and $|L_\alpha \cap \Sigma_\infty| \geq \mathfrak{c}$. On one hand, we have
\[
|L_\alpha \cap \Sigma_\infty| \leq |\Sigma| \leq |\mathbb{N}_\infty^{\mathbb{N}}| = |\mathbb{N}^{\mathbb{N}}| = \mathfrak{c}.
\]
Here, the inequalities follow from the inclusions $L_\alpha \cap \Sigma_\infty \subseteq \Sigma \subseteq \mathbb{N}_\infty^{\mathbb{N}}$. For the first equality, it is enough to notice that there is an obvious bijection between $\mathbb{N}_\infty^{\mathbb{N}}$ and $\mathbb{N}^{\mathbb{N}}$. The second equality is classical.

For the reverse inequality, we construct an injection from $\{ 0, 1 \}^{\mathbb{N}}$ into $L_\alpha \cap \Sigma_\infty$. For each $\epsilon \coloneqq (\epsilon_k)_{k \in \mathbb{N}} \in \{ 0, 1 \}^{\mathbb{N}}$, define a sequence $\sigma_\epsilon$ by
\[
\sigma_\epsilon
\coloneqq
\begin{cases}
 ( (\epsilon_1+1)^1, (\epsilon_2+3)^2, \dotsc, (\epsilon_k+2k-1)^k, \dotsc), &\text{if } \alpha = 0; \\
 ( \lfloor (\epsilon_1+1)^{1/\alpha} \rfloor, \lfloor (\epsilon_2+3)^{1/\alpha} \rfloor, \dotsc, \lfloor (\epsilon_k+2k-1)^{1/\alpha} \rfloor, \dotsc), &\text{if } \alpha \in (0,1].
\end{cases}
\]
It is clear that $\sigma_\epsilon \in \Sigma_\infty$ and that the mapping $\epsilon \mapsto \sigma_\epsilon$ is injective. Moreover, since $\epsilon_k \in \{0, 1\}$ for each $k \in \mathbb{N}$, we have
\[
\lambda (\sigma_\epsilon) = 
\begin{dcases}
\limsup_{k \to \infty} \frac{\log k}{\log (\epsilon_k+2k-1)^k} = 0, &\text{if } \alpha = 0; \\
\limsup_{k \to \infty} \frac{\log k}{\log \lfloor (\epsilon_k+2k-1)^{1/\alpha} \rfloor} = \alpha, &\text{if } \alpha \in (0,1],
\end{dcases}
\]
by using the formula \eqref{PS formula on Sigma}. This shows that $\sigma_\epsilon \in L_\alpha$, and hence $\sigma_\epsilon \in L_\alpha \cap \Sigma_\infty$. Therefore, by the well-known fact $|\{ 0, 1 \}^{\mathbb{N}}| = \mathfrak{c}$, we find that
\[
\mathfrak{c} = |\{ 0, 1 \}^{\mathbb{N}}| \leq |L_\alpha \cap \Sigma_\infty|.
\]
This completes the proof.
\end{proof}

The rest of this section is dedicated to the map $\Psi^{(s)} \colon \Sigma \to [0, \infty]$, for each $s \in (0,1]$, defined by
\[
\Psi^{(s)} (\sigma) \coloneqq \sum_{k \in \mathbb{N}} \frac{1}{\sigma_k^{s}}
\]
for each $\sigma \coloneqq (\sigma_k)_{k \in \mathbb{N}} \in \Sigma$.

\begin{lemma} \label{Psi alpha lower semicontinuity lemma}
For each $s \in (0,1]$, the mapping $\Psi^{(s)} \colon \Sigma \to [0, \infty]$ is lower semicontinuous.
\end{lemma}

\begin{proof}
Let $s \in (0,1]$. We show that $\{ \sigma \in \Sigma : \Psi^{(s)} (\sigma) > y \}$ is open in $\Sigma$ for each $y \in \mathbb{R}$. Note first that
\[
\{ \sigma \in \Sigma : \Psi^{(s)} (\sigma) > y \}
= \begin{cases}
\Sigma, &\text{if } y < 0; \\
\Sigma \setminus \{ (\infty, \infty, \dotsc) \}, &\text{if } y = 0,
\end{cases}
\]
where $\Sigma$ is vacuously open and the latter set, as the complement of a singleton, is open since $\Sigma$ is a metric space (Proposition \ref{Sigma is compact metric}). Now, fix $y \in (0, \infty)$. Let $\sigma \coloneqq (\sigma_k)_{k \in \mathbb{N}} \in \{ \sigma \in \Sigma : \Psi^{(s)} (\sigma) > y \}$ be arbitrary. Then, we can find an $M \in \mathbb{N}$ such that
\[
\sum_{k=1}^{M-1} \frac{1}{\sigma_k^{s}} \leq y \quad \text{and} \quad \sum_{k=1}^M \frac{1}{\sigma_k^{s}} > y.
\]
Clearly, $\sigma_M \neq \infty$, which, in conjunction with the definition of $\Sigma$, tells us that $\sigma_k \in \mathbb{N}$ for all $1 \leq k \leq M$. Consider the open set
\[
U \coloneqq \Sigma \cap \left( \prod_{k=1}^M \{ \sigma_k \} \times \prod_{k > M} \mathbb{N}_\infty \right)
\] 
in $\Sigma$. It is obvious that $\sigma \in U$. Moreover, for any $\tau \coloneqq (\tau_k)_{k \in \mathbb{N}} \in U$, we have
\begin{align*}
\Psi^{(s)} (\tau) 
= \sum_{k=1}^M \frac{1}{\tau_k^{s}} + \sum_{k > M} \frac{1}{\tau_k^{s}}
= \sum_{k=1}^M \frac{1}{\sigma_k^{s}} + \sum_{k > M} \frac{1}{\tau_k^{s}}
> y.
\end{align*}
Thus, $U \subseteq \{ \sigma \in \Sigma : \Psi^{(s)} (\sigma) > y \}$, and this completes the proof.
\end{proof}

\begin{remark} \label{Psi alpha not upper semicontinuous}
We remark that $\Psi^{(s)} \colon \Sigma \to [0,\infty]$ fails to be upper semicontinuous for any $s \in (0,1]$. To see this, fix $s \in (0,1]$. Let $\sigma \coloneqq (\sigma_k)_{k \in \mathbb{N}} \in \Sigma_\infty$ with $M \coloneqq \Psi^{(s)} (\sigma) < \infty$, and consider a sequence $(\bm{\tau}_j)_{j \in \mathbb{N}}$ in $\Sigma$ given by
\[
\bm{\tau}_j \coloneqq (\sigma_1, \dotsc, \sigma_j, \lfloor (\sigma_j+1)^{1/s} \rfloor, \lfloor (\sigma_j+2)^{1/s} \rfloor, \dotsc)
\]
for each $j \in \mathbb{N}$. Then, $\bm{\tau}_j \to \sigma$ as $j \to \infty$, and $\Psi^{(s)} (\bm{\tau}_j) = \infty > M$ for all $j \in \mathbb{N}$. This, in particular, shows that the superlevel set $\{ \sigma \in \Sigma : \Psi^{(s)} (\sigma) \geq 2M \}$ is not closed in $\Sigma$.
\end{remark}

\begin{lemma} \label{density lemma 1}
For each $s \in (0,1]$, the $\infty$-level set $L_\infty (\Psi^{(s)}) \coloneqq \{ \sigma \in \Sigma : \Psi^{(s)} (\sigma) = \infty \}$ is dense $G_\delta$ in $\Sigma_\infty$.
\end{lemma}

\begin{proof}
Let $s \in (0,1]$. We first note that $L_\infty (\Psi^{(s)}) \subseteq \Sigma_\infty$. To see this, let $\sigma \coloneqq (\sigma_j)_{j \in \mathbb{N}} \in \Sigma \setminus \Sigma_\infty$. Then, we can find a $K \in \mathbb{N}$ such that $\sigma_j = \infty$, hence $1/\sigma_j^s = 0$ due to the convention, for all $j > K$. This implies that $\Psi^{(s)}(\sigma) = \sum_{j=1}^{K} 1/\sigma_j^s$ is finite, i.e., $\sigma \in \Sigma \setminus L_\infty (\Psi^{(s)})$.

Now, the argument in Remark \ref{Psi alpha not upper semicontinuous} proves the denseness of $L_\infty (\Psi^{(s)})$ in $\Sigma_\infty$. Next, by writing
\begin{align*}
L_\infty (\Psi^{(s)}) = \bigcap_{n=1}^\infty \{ \sigma \in \Sigma : \Psi^{(s)} (\sigma) > n \} = \bigcap_{n=1}^\infty [\Sigma_\infty \cap \{ \sigma \in \Sigma : \Psi^{(s)} (\sigma) > n \}],
\end{align*}
we deduce that $L_\infty (\Psi^{(s)})$ is $G_\delta$ in $\Sigma_\infty$ since each $\{ \sigma \in \Sigma : \Psi^{(s)} (\sigma) > n \}$ is open in $\Sigma$ by Lemma \ref{Psi alpha lower semicontinuity lemma}.
\end{proof}

\section{Proofs of main results} \label{Proofs of main results}

In this section, we prove the main results of this paper mentioned in Section \ref{Introduction}.

\subsection{Proofs of results on the convergence exponent}

\begin{proof} [Proof of Theorem \ref{lambda star is zero a.e.}]
The law of large numbers in Pierce expansions \cite[Theorem 16]{Sha86} states that $(d_n(x))^{1/n} \to e$ as $n \to \infty$ Lebesgue-almost everywhere on $[0,1]$. Thus, for Lebesgue-almost every $x \in [0,1]$, we have $(1/(d_n(x))^s)^{1/n} \to 1/e^s < 1$ as $n \to \infty$ for any $s>0$. Hence, the theorem follows from the root test.
\end{proof}

\begin{lemma} \label{basic results on lambda star}
For $\lambda^* \colon [0,1] \to [0, \infty]$, the following hold:
\begin{enumerate}[label=\upshape(\roman*), ref=\roman*, leftmargin=*, widest=ii]
\item \label{basic results on lambda star 1}
For any rational $x \in [0,1]$, we have $\lambda^* (x) = 0$.
\item \label{basic results on lambda star 2}
$\lambda^* ([0,1]) \subseteq [0,1]$.
\end{enumerate}
\end{lemma}

\begin{proof}
Both parts follow from \eqref{lambda star as composition}, Proposition \ref{preimage of phi}, and Lemma \ref{basic results on lambda}. In fact, we have
\[
\lambda^*(x) = \lambda (f(x)) \in \lambda \left( \bigcup_{n \in \mathbb{N}_0} \Sigma_n \right) = \{ 0 \}
\]
for any rational $x \in [0,1]$, and
\[
\lambda^*([0,1]) = \lambda (f([0,1])) \subseteq \lambda (\Sigma) \subseteq [0,1],
\]
as desired.
\end{proof}

\begin{proof} [Proof of Theorem \ref{lambda star is surjective}]
Let $U$ be a non-empty open subset of $[0,1]$. Since $U \cap \mathbb{I}$ is open in $\mathbb{I}$ and $f|_{\mathbb{I}} \colon \mathbb{I} \to \Sigma_\infty$ is a homeomorphism (Proposition \ref{f is homeo}(\ref{f is homeo 3})), we have 
\[
\lambda^*(U \setminus \mathbb{Q}) = \lambda^*(U \cap \mathbb{I}) = (\lambda \circ f)(U \cap \mathbb{I}) = \lambda ( f|_{\mathbb{I}}(U \cap \mathbb{I}) ) = \lambda (V \cap \Sigma_\infty)
\]
for some non-empty open subset $V$ of $\Sigma$, where we used \eqref{lambda star as composition} for the second equality. But $\lambda (V \cap \Sigma_\infty) = [0,1]$ by Theorem \ref{neighborhood image theorem}. Thus, we obtain
\[
[0,1] \supseteq \lambda^*(U) \supseteq \lambda^*(U \cap \mathbb{I}) = [0,1],
\]
where the first inclusion follows from Lemma \ref{basic results on lambda star}(\ref{basic results on lambda star 1}). This finishes the proof.
\end{proof}

\begin{proof} [Proof of Corollary \ref{lambda star has intermediate value property}]
The corollary is immediate from Theorem \ref{lambda star is surjective}.
\end{proof}

\begin{proof} [Proof of Corollary \ref{lambda star discontinuous corollary}]
The corollary is another immediate consequence of Theorem \ref{lambda star is surjective}.
\end{proof}

\begin{proof} [Proof of Corollary \ref{lambda star not first Baire corollary}]
Since $\lambda^* \colon [0,1] \to [0,1]$ is discontinuous everywhere (Corollary \ref{lambda star discontinuous corollary}) on the complete metric space $[0,1]$, Proposition \ref{not first Baire proposition} implies that $\lambda^*$ is not of Baire class one.
\end{proof}

\begin{lemma} \label{PS formula on unit interval lemma}
For any $x \in [0,1]$, we have
\begin{align} \label{PS formula on unit interval}
\lambda^* (x) = \limsup_{n \to \infty} (\psi_n \circ f)(x) = \limsup_{n \to \infty} \frac{\log n}{\log d_n(x)}.
\end{align}
\end{lemma}

\begin{proof}
The desired formula follows on combining \eqref{lambda star as composition}, \eqref{definition of psi n}, and \eqref{PS formula on Sigma}.
\end{proof}

\begin{proof} [Proof of Theorem \ref{lambda star second Baire}]
We claim that
\[
\lambda^*(x) = \inf_{\upsilon \in \varphi^{-1}( \{ x \})} \left( \limsup_{n \to \infty} \psi_n (\upsilon) \right)
\]
for each $x \in [0,1]$. In fact, for each $x \in \mathbb{I} \cup \{ 0, 1 \}$, since $\varphi^{-1}(\{ x \}) = \{ f(x) \}$ by Proposition \ref{preimage of phi}(\ref{preimage of phi 1}), it follows that
\begin{align*}
\inf_{\upsilon \in \varphi^{-1}( \{ x \})} \left( \limsup_{n \to \infty} \psi_n (\upsilon) \right) 
&= \limsup_{n \to \infty} \psi_n (f(x)) = \lambda^*(x),
\end{align*}
where we used \eqref{PS formula on unit interval} for the second equality. Assume $x \in \mathbb{Q} \cap (0,1)$. By Proposition \ref{preimage of phi}(\ref{preimage of phi 2}), we know that $\varphi^{-1}(\{ x \})$ contains two elements, say $\sigma$ and $\tau$, both of which are in $\bigcup_{n \in \mathbb{N}} \Sigma_n$. Hence,
\begin{align*}
\inf_{\upsilon \in \varphi^{-1}( \{ x \})} \left( \limsup_{n \to \infty} \psi_n (\upsilon) \right)
&= \min \left\{ \limsup_{n \to \infty} \psi_n (\sigma), \limsup_{n \to \infty} \psi_n (\tau) \right\} \\
&= \min \{ \lambda (\sigma), \lambda (\tau) \}
= 0,
\end{align*}
where the second equality follows from \eqref{PS formula on Sigma} and the third from Lemma \ref{basic results on lambda}(\ref{basic results on lambda 1}). But, since $x$ is rational, we have $\lambda^*(x) = 0$ by Lemma \ref{basic results on lambda star}(\ref{basic results on lambda star 1}). This proves the claim. 

Recall that $\varphi \colon \Sigma \to [0,1]$ is continuous (Proposition \ref{f is homeo}(\ref{f is homeo 2})) and surjective (Proposition \ref{preimage of phi}), and that $\Sigma$ and $[0,1]$ are compact metric spaces (see Proposition \ref{Sigma is compact metric} for $\Sigma$). We know, by the definition \eqref{definition of psi n}, that each $\psi_n \colon \Sigma \to [0,1]$ is continuous, hence lower semicontinuous. Therefore, since $\lambda^*$ is real-valued by Lemma \ref{basic results on lambda star}(\ref{basic results on lambda star 2}), we conclude, in view of the claim above and Lemma \ref{main lemma}, that $\lambda^*$ is of Baire class two on $[0,1]$.
\end{proof}

\subsection{Proofs of results on the level sets of the convergence exponent}

The calculation method used in the proof of the following lemma is typical in fractal studies. See, e.g., \cite[Lemma 4.10]{Ahn23b} and \cite[Lemma 4.3]{SW21}.

\begin{lemma} \label{hidm upper bound lemma}
For each $\alpha, \beta \in (0,1]$ with $\alpha \leq \beta$, we have
\begin{align} \label{hidm upper bound formula}
\hdim \bigcup_{\gamma \in [\alpha, \beta]} L_\gamma^* \leq \beta ( \alpha^{-1} - 1 ).
\end{align}
\end{lemma}

\begin{proof}
Fix $\alpha, \beta \in (0,1]$ such that $\alpha \leq \beta$. Let $\varepsilon \in (0, \alpha)$ be given. Put
\begin{align*}
B_N (\varepsilon) \coloneqq \bigcap_{m \geq N} \bigcup_{k \geq m} \{ x \in (0,1] : d_j(x) \geq j^{(\beta+\varepsilon)^{-1}} \text{ for all } &N \leq j \leq k, \\
&\text{ and } d_k(x) \leq k^{(\alpha-\varepsilon)^{-1}} \}
\end{align*}
and
\[
B(\varepsilon) \coloneqq \bigcup_{N \in \mathbb{N}} B_N(\varepsilon).
\]
Observe that $\bigcup_{\gamma \in [\alpha, \beta]} L_\gamma^* \subseteq B(\varepsilon)$. In fact, if $x \in \bigcup_{\gamma \in [\alpha, \beta]} L_\gamma^*$, then the following hold by \eqref{PS formula on unit interval}:
\begin{enumerate}[label=\upshape(\roman*), ref=\roman*, leftmargin=*, widest=ii]
\item
$\log j/\log d_j(x) \geq \alpha - \varepsilon$, or $d_j(x) \leq j^{(\alpha-\varepsilon)^{-1}}$, for infinitely many $j \in \mathbb{N}$.
\item
$\log j/\log d_j(x) \leq \beta + \varepsilon$, or $d_j(x) \geq j^{(\beta+\varepsilon)^{-1}}$, for all sufficiently large $j \in \mathbb{N}$.
\end{enumerate}
Hence, there exists $N \in \mathbb{N}$ such that for any $m \geq N$, we can find $k \geq m$ such that $d_j(x) \geq j^{(\beta+\varepsilon)^{-1}}$ for all $N \leq j \leq k$ and $d_k(x) \leq k^{(\alpha-\varepsilon)^{-1}}$. By monotonicity and countable stability of the Hausdorff dimension (see \cite[p.~48--49]{Fal14}), respectively, we have
\begin{align} \label{hdim L alpha star leq hdim B epsilon}
\hdim \bigcup_{\gamma \in [\alpha, \beta]} L_\gamma^* \leq \hdim B(\varepsilon) = \sup_{N \in \mathbb{N}} \hdim \{ B_N (\varepsilon) \}.
\end{align}

Let $N \in \mathbb{N}$ be fixed, and consider $B_N (\varepsilon)$. For each integer $k \geq N$, let
\[
\Lambda_k \coloneqq \{ (\sigma_j)_{j \in \mathbb{N}} \in \Sigma_k : \sigma_j \geq j^{(\beta+\varepsilon)^{-1}} \text{ for all } N \leq j \leq k \text{ and } \sigma_k \leq k^{(\alpha-\varepsilon)^{-1}} \}.
\]
Notice that, for any $x \in [0,1]$, we have the following equivalences:
\begin{align*}
&d_j(x) \geq j^{(\beta+\varepsilon)^{-1}} \text{ for all } N \leq j \leq k, \text{ and } d_k(x) \leq k^{(\alpha-\varepsilon)^{-1}} \\
&\hspace{2cm} \iff f(x) \in \Upsilon_\sigma \text{ for some } \sigma \in \Lambda_k \\
&\hspace{2cm} \iff x \in I_\sigma \text{ for some } \sigma \in \Lambda_k.
\end{align*}
Here, the second equivalence holds by the definition $I_\sigma = f^{-1}(\Upsilon_\sigma)$ in \eqref{definition of I sigma}. Then,
\begin{align} \label{equivalent definition of E ln rn}
B_N(\varepsilon) = \bigcap_{m \geq N} \bigcup_{k \geq m} \bigcup_{\sigma \in \Lambda_k} I_\sigma.
\end{align}
For each $\sigma \coloneqq (\sigma_j)_{j \in \mathbb{N}} \in \Lambda_k$, using the equation \eqref{diam I sigma} and the fact that $\sigma_n \geq n$ for each $n \in \mathbb{N}$, we obtain an upper bound for $\diam I_\sigma$ as follows:
\begin{align} \label{bound for diam I sigma}
\diam I_\sigma 
= \left( \prod_{j=1}^{k} \frac{1}{\sigma_j} \right) \frac{1}{\sigma_k+1} 
&\leq \frac{1}{(N-1)!} \frac{1}{[N(N+1) \dotsm k]^{(\beta+\varepsilon)^{-1}}(k^{(\beta+\varepsilon)^{-1}}+1)} \nonumber \\
&\leq \frac{1}{[N(N+1) \dotsm k]^{(\beta+\varepsilon)^{-1}}} 
= \left( \frac{(N-1)!}{k!} \right)^{(\beta+\varepsilon)^{-1}}.
\end{align}
Moreover, for each integer $k \geq N$, since any sequence in $\Lambda_k$ is strictly increasing, we have
\begin{align} \label{bound for card Lambda k}
| \Lambda_k |
\leq {\lfloor k^{(\alpha-\varepsilon)^{-1}} \rfloor \choose k} 
&= \frac{\lfloor k^{(\alpha-\varepsilon)^{-1}} \rfloor (\lfloor k^{(\alpha-\varepsilon)^{-1}} \rfloor-1) \dotsm (\lfloor k^{(\alpha-\varepsilon)^{-1}} \rfloor-(k-1))}{k!} \nonumber \\
&\leq \frac{k^{(\alpha-\varepsilon)^{-1}} (k^{(\alpha-\varepsilon)^{-1}}-1) \dotsm (k^{(\alpha-\varepsilon)^{-1}}-(k-1))}{k!} \nonumber \\
&\leq \frac{k^{k(\alpha-\varepsilon)^{-1}}}{k!}.
\end{align}

Now, let $s > (\beta+\varepsilon) ( (\alpha-\varepsilon)^{-1} - 1 )$. By \eqref{equivalent definition of E ln rn}, for each $m \geq N$, the collection $\bigcup_{k \geq m} \{ I_\sigma : \sigma \in \Lambda_k \}$ is a covering of $B_N(\varepsilon)$. Then, by using \eqref{bound for diam I sigma} and \eqref{bound for card Lambda k}, we find that
\begin{align*}
\mathcal{H}^s (B_N (\varepsilon))
&\leq \liminf_{m \to \infty} \sum_{k \geq m} \sum_{\sigma \in \Lambda_k} (\diam I_\sigma)^s \\
&\leq \liminf_{m \to \infty} \sum_{k \geq m} \left[ \frac{k^{k(\alpha-\varepsilon)^{-1}}}{k!} \left( \frac{(N-1)!}{k!} \right)^{s(\beta+\varepsilon)^{-1}} \right] \\
&= [(N-1)!]^{s(\beta+\varepsilon)^{-1}} \liminf_{m \to \infty} \sum_{k \geq m} \frac{k^{k(\alpha-\varepsilon)^{-1}}}{(k!)^{s(\beta+\varepsilon)^{-1}+1}}.
\end{align*}
Put $a_k \coloneqq k^{k(\alpha-\varepsilon)^{-1}} / (k!)^{s(\beta+\varepsilon)^{-1}+1}$ for each $k \in \mathbb{N}$. Then, since $s > (\beta+\varepsilon) ( (\alpha-\varepsilon)^{-1} - 1 )$, or, equivalently, $(\alpha-\varepsilon)^{-1} - s(\beta+\varepsilon)^{-1} - 1 < 0$, it follows that
\begin{align*}
\frac{a_{k+1}}{a_k}
&= {\dfrac{(k+1)^{k(\alpha-\varepsilon)^{-1}}(k+1)^{(\alpha-\varepsilon)^{-1}}}{(k!)^{s(\beta+\varepsilon)^{-1}+1}(k+1)^{s(\beta+\varepsilon)^{-1}+1}}} \bigg/ {\dfrac{k^{k(\alpha-\varepsilon)^{-1}}}{(k!)^{s(\beta+\varepsilon)^{-1}+1}}} \\
&= \left[ \left( 1 + \frac{1}{k} \right)^k \right]^{(\alpha-\varepsilon)^{-1}} (k+1)^{(\alpha-\varepsilon)^{-1} - s(\beta+\varepsilon)^{-1}-1}
\to e^{(\alpha-\varepsilon)^{-1}} \cdot 0 = 0
\end{align*}
as $k \to \infty$. This tells us that the series $\sum_{k \in \mathbb{N}} a_k$ is convergent, and so $\sum_{k \geq m} a_k \to 0$ as $m \to \infty$. Therefore, $\mathcal{H}^s (B_N(\varepsilon)) = 0$, which implies $\hdim B_N(\varepsilon) \leq s$ by Definition \ref{hdim definition}. Since $s \in \left( (\beta+\varepsilon) ( (\alpha-\varepsilon)^{-1} - 1 ), \infty \right)$ was arbitrary, it follows that $\hdim B_N(\varepsilon) \leq (\beta+\varepsilon) ( (\alpha-\varepsilon)^{-1} - 1 )$. But $N \in \mathbb{N}$ was arbitrary as well, so that \eqref{hdim L alpha star leq hdim B epsilon} gives
\[
\hdim \bigcup_{\gamma \in [\alpha, \beta]} L_\gamma^* \leq (\beta+\varepsilon) ( (\alpha-\varepsilon)^{-1} - 1 ).
\]
Finally, by letting $\varepsilon \to 0^+$, we deduce that $\hdim \bigcup_{\gamma \in [\alpha, \beta]} L_\gamma^* \leq \beta ( \alpha^{-1} - 1 )$, as was to be shown.
\end{proof}

The following lemma and its proof are inspired by \cite[Theorem 4.2]{SW21}.

\begin{lemma} \label{hdim formula lemma}
For each $\alpha, \beta \in (0,1]$ with $\alpha \leq \beta$, we have
\begin{align} \label{hdim formula}
\hdim \bigcup_{\gamma \in [\alpha, \beta]} L_\gamma^*  = \hdim L_\alpha^* = 1 - \alpha.
\end{align}
\end{lemma}

\begin{proof}
Let $\alpha, \beta \in (0,1]$ with $\alpha \leq \beta$. It is clear from the definitions and \eqref{PS formula on unit interval} that $E(\alpha) \subseteq L_\alpha^* \subseteq \bigcup_{\gamma \in [\alpha, \beta]} L_\gamma^*$, where $E(\alpha)$ is defined as in Proposition \ref{hdim E alpha}. Then, Proposition \ref{hdim E alpha}, in conjunction with monotonicity of the Hausdorff dimension (see \cite[p.~48]{Fal14}), implies that 
\begin{align} \label{auxiliary inequalities}
1 - \alpha = \hdim E(\alpha) \leq \hdim L_\alpha^* \leq \hdim \bigcup_{\gamma \in [\alpha, \beta]} L_\gamma^*.
\end{align}
If $\alpha = \beta$, then, by \eqref{hidm upper bound formula}, the rightmost term in \eqref{auxiliary inequalities} is bounded above by $\alpha (\alpha^{-1}-1) = 1-\alpha$, and thus \eqref{hdim formula} holds true.

Now, assume $\alpha < \beta$. Let $n \in \mathbb{N}$. Put $\Delta \coloneqq (\beta-\alpha)/n > 0$. Write
\[
\bigcup_{\gamma \in [\alpha, \beta]} L_\gamma^* = \bigcup_{j=0}^{n-1} \bigcup_{\gamma \in [\alpha_j, \alpha_{j+1}]} L_\gamma^*, \quad \text{where } \alpha_j \coloneqq \alpha + j \Delta \text{ for } j \in \{ 0, 1, \dotsc, n \}.
\]
Then, by countable stability of the Hausdorff dimension (see \cite[p.~48]{Fal14}) and \eqref{hidm upper bound formula}, we have
\begin{align*}
\hdim \bigcup_{\gamma \in [\alpha, \beta]} L_\gamma^* 
&= \max_{0 \leq j \leq n-1} \left\{ \hdim \bigcup_{\gamma \in [\alpha_j, \alpha_{j+1}]} L_\gamma^* \right\} 
\leq \max_{0 \leq j \leq n-1} \left\{ \alpha_{j+1} ( \alpha_j^{-1} - 1) \right\}.
\end{align*}
Observe that for each $j \in \{ 0, \dotsc, n-2 \}$, we have
\begin{align*}
\alpha_{j+1} ( \alpha_{j}^{-1}-1 ) - \alpha_{j+2} ( \alpha_{j+1}^{-1} - 1 )
&= \frac{\alpha_{j+1}^2 - (\alpha_{j+1} - \Delta) (\alpha_{j+1} + \Delta)}{\alpha_j \alpha_{j+1}} + \alpha_{j+2}-\alpha_{j+1} \\
&= \frac{\Delta^2}{\alpha_j \alpha_{j+1}} + \Delta > 0.
\end{align*}
It follows that
\[
\hdim \bigcup_{\gamma \in [\alpha, \beta]} L_\gamma^* \leq \alpha_1 ( \alpha_0^{-1} - 1 ) = \left( \alpha + \frac{\beta-\alpha}{n} \right) ( \alpha^{-1} - 1 ).
\]
Notice that the leftmost term of the equation above is independent of $n$. Therefore, by letting $n \to \infty$, we find that
\[
\hdim \bigcup_{\gamma \in [\alpha, \beta]} L_\gamma^* \leq (\alpha + 0) ( \alpha^{-1} - 1 ) = 1 - \alpha.
\]
On combining the above inequality with \eqref{auxiliary inequalities}, we conclude that \eqref{hdim formula} holds true.
\end{proof}

\begin{proof} [Proof of Theorem \ref{hdim L alpha star theorem}]
As we mentioned in the paragraph preceding the statement of the theorem, we have $\hdim L_0^* = 1$. On the other hand, for each $\alpha \in (0,1]$, we have $\hdim L_\alpha^* = 1 - \alpha$ from Lemma \ref{hdim formula lemma}. Hence the theorem.
\end{proof}

\begin{lemma} \label{L alpha is f image of L alpha star}
Let $\alpha \in [0,1]$. For the $\alpha$-level sets $L_\alpha$ and $L_\alpha^*$, we have $L_\alpha \cap \Sigma_\infty = f|_{\mathbb{I}} (L_\alpha^* \cap \mathbb{I})$, or, equivalently, $L_\alpha^* \cap \mathbb{I} = \varphi|_{\Sigma_\infty} (L_\alpha \cap \Sigma_\infty)$. In particular, if $\alpha \neq 0$, then $L_\alpha = f|_{\mathbb{I}} (L_\alpha^*)$, or, equivalently, $L_\alpha^* = \varphi|_{\Sigma_\infty} (L_\alpha)$.
\end{lemma}

\begin{proof}
Throughout the proof, we frequently make use of the fact that $f|_{\mathbb{I}} \colon \mathbb{I} \to \Sigma_\infty$ is a homeomorphism with the continuous inverse $\varphi|_{\Sigma_\infty} \colon \Sigma_\infty \to \mathbb{I}$ (Proposition \ref{f is homeo}(\ref{f is homeo 3})).

In the first assertion, since $L_\alpha \cap \Sigma_\infty \subseteq \Sigma_\infty$ and $L_\alpha^* \cap \mathbb{I} \subseteq \mathbb{I}$, the equivalence part is clear. We only need to show that $L_\alpha \cap \Sigma_\infty = f|_{\mathbb{I}} (L_\alpha^* \cap \mathbb{I})$.

Suppose $\sigma \in L_\alpha \cap \Sigma_\infty$. Then, $\lambda (\sigma) = \alpha$ by definition of $L_\alpha$. Since $\sigma \in \Sigma_\infty$, we can find a unique $x \in \mathbb{I}$ such that $f|_{\mathbb{I}}(x) = \sigma$. By using \eqref{lambda star as composition}, we find that $\lambda^*(x) = (\lambda \circ f)(x) = \lambda (\sigma) = \alpha$. Thus, $x \in L_\alpha^* \cap \mathbb{I}$, and so, $\sigma = f|_{\mathbb{I}}(x) \in f|_{\mathbb{I}}(L_\alpha^* \cap \mathbb{I})$. This proves the inclusion $L_\alpha \cap \Sigma_\infty \subseteq f|_{\mathbb{I}}(L_\alpha^* \cap \mathbb{I})$.

Conversely, suppose $\sigma \in f|_{\mathbb{I}}(L_\alpha^* \cap \mathbb{I})$. Then, $\sigma \in \Sigma_\infty$, and so, there is a unique $x \in \mathbb{I}$ such that $\sigma = f|_{\mathbb{I}} (x)$. Now, we have $x = \varphi|_{\Sigma_\infty} (\sigma) \in L_\alpha^* \cap \mathbb{I}$, so that $\lambda^*(x) = \alpha$. It follows from \eqref{lambda star as composition} that $\lambda (\sigma) = \lambda (f(x)) = \lambda^*(x) = \alpha$. Thus, $\sigma \in L_\alpha \cap \Sigma_\infty$, and this proves the other inclusion $L_\alpha \cap \Sigma_\infty \supseteq f|_{\mathbb{I}}(L_\alpha^* \cap \mathbb{I})$.

The second assertion follows by noting that whenever $\alpha \in (0,1]$, we have $L_\alpha \cap \Sigma_\infty = L_\alpha$ and $L_\alpha^* \cap \mathbb{I} = L_\alpha^*$ by Lemmas \ref{basic results on lambda}(\ref{basic results on lambda 1}) and \ref{basic results on lambda star}(\ref{basic results on lambda star 1}), respectively.
\end{proof}

\begin{proof} [Proof of Theorem \ref{L alpha star dense G delta theorem}]
The fact that the $\alpha$-level set $L_\alpha^*$ is dense in $[0,1]$ for each $\alpha \in [0,1]$ immediately follows from Theorem \ref{lambda star is surjective}.

Next, recall that $L_1^* = \varphi|_{\Sigma_\infty} (L_1)$ by Lemma \ref{L alpha is f image of L alpha star} and that $L_1$ is $G_\delta$ in $\Sigma_\infty$ by Theorem \ref{L alpha dense G delta theorem}. But $\varphi|_{\Sigma_\infty} \colon \Sigma_\infty \to \mathbb{I}$ is a homeomorphism by Proposition \ref{f is homeo}(\ref{f is homeo 3}), so that $L_1^*$ is $G_\delta$ in $\mathbb{I}$. Since $\mathbb{I} = \bigcap_{q \in [0,1] \cap \mathbb{Q}} ([0,1] \setminus \{ q \})$ is $G_\delta$ in $[0,1]$, we conclude that $L_1^*$ is $G_\delta$ in $[0,1]$.

Now, we write the complete metric space $[0,1]$ as a disjoint union of the collection $\{ L_\alpha^* : \alpha \in [0,1] \}$. Due to the preceding two paragraphs, we know that $L_1^*$ is dense $G_\delta$ in $[0,1]$. Then, Lemma \ref{at most one G delta level set} tells us that $L_1^*$ is the unique dense $G_\delta$ element in the collection $\{ L_\alpha^* : \alpha \in [0,1] \}$. Thus, for each $\alpha \in [0,1)$, since $L_\alpha^*$ is dense in $[0,1]$ by the first paragraph of this proof, we find that $L_\alpha^*$ fails to be $G_\delta$ in $[0,1]$.

Lastly, we verify the second statement of the theorem. Since $L_1^*$ is dense $G_\delta$ in the complete metric space $[0,1]$, Fact \ref{basic facts}(\ref{basic facts 1}) tells us that $L_1^*$ is comeager and nonmeager in $[0,1]$. Then, for each $\alpha \in [0,1)$, the $\alpha$-level set $L_\alpha^*$ is a subset of $[0,1] \setminus L_1^*$, which is meager in $[0,1]$, and thus, $L_\alpha^*$ is meager in $[0,1]$ as well. This completes the proof of the theorem.
\end{proof}

\begin{proof} [Proof of Corollary \ref{bdim L alpha star corollary}]
For each $\alpha \in [0,1]$, since $L_\alpha^*$ is dense in $[0,1]$ by Theorem \ref{L alpha star dense G delta theorem}, we deduce that $\bdim L_\alpha^*=1$ in view of Proposition \ref{dense bdim proposition}.
\end{proof}

\begin{proof} [Proof of Theorem \ref{L alpha star is uncountable}]
Let $\alpha \in [0,1]$. On one hand, we have
\[
\mathfrak{c} = | L_\alpha  \cap \Sigma_\infty | = | \varphi|_{\Sigma_\infty} (L_\alpha \cap \Sigma_\infty)| = | L_\alpha^* \cap \mathbb{I} | \leq |L_\alpha^*|,
\]
where the first equality follows from Lemma \ref{cardinality of L alpha lemma}, the second from Proposition \ref{f is homeo}(\ref{f is homeo 3}), the third from Lemma \ref{L alpha is f image of L alpha star}, and the inequality from the inclusion $L_\alpha^*\cap \mathbb{I} \subseteq L_\alpha^*$. 
On the other hand, due to the inclusion $L_\alpha^* \subseteq [0,1]$ and the classical fact $|[0,1]| = \mathfrak{c}$, we have
\[
|L_\alpha^*| \leq | [0,1] | = \mathfrak{c}.
\]
Therefore, we infer from the Cantor-Schr{\"o}der-Bernstein theorem that $|L_\alpha^*| = \mathfrak{c}$.
\end{proof}

\subsection{Proofs of results on the series of positive $s$th powers of the reciprocals of the digits}

\begin{lemma} \label{relation of Theta and L alpha star}
For each $s \in (0,1]$, we have
\[
\bigcup_{\alpha \in (s,1]} L_\alpha^* \subseteq \mathcal{D}^{(s)}_{\Div} \subseteq \bigcup_{\alpha \in [s,1]} L_\alpha^* \subseteq \mathbb{I}.
\]
\end{lemma}

\begin{proof}
If $s = 1$, the first inclusion is evident since the union on the left is empty. Fix $s \in (0,1)$. Let $\alpha \in (s,1]$ be arbitrary. Observe, by definition \eqref{definition of lambda star} of the convergence exponent, that if $\lambda^*(x) = \alpha$ for some $x \in [0,1]$, then the series $\sum_{n \in \mathbb{N}} 1/(d_n(x))^s$ is divergent, i.e., $x \in \mathcal{D}^{(s)}_{\Div}$ by the definition \eqref{definition of Theta}. This proves the first inclusion.

Let $s \in (0,1]$. Suppose $x \in \mathcal{D}^{(s)}_{\Div}$. Then, the series $\sum_{n \in \mathbb{N}} 1/(d_n(x))^s$ diverges by the definition \eqref{definition of Theta}, and so $\lambda^*(x) \geq s$ by the definition \eqref{definition of lambda star}. But $\lambda^*(x) \leq 1$ by Lemma \ref{basic results on lambda star}(\ref{basic results on lambda star 2}). Hence, $\lambda^*(x) \in [s,1]$, or, equivalently, $x \in \bigcup_{\alpha \in [s,1]} L_\alpha^*$. This proves the second inclusion.

For the last inclusion, it is enough to note that $L_\alpha^* \subseteq \mathbb{I}$ for each $\alpha \in (0,1]$ by Lemma \ref{basic results on lambda star}(\ref{basic results on lambda star 1}).
\end{proof}

\begin{proof} [Proof of Theorem \ref{Theta Lebesgue measure theorem}]
Let $s \in (0,1]$. Then, by Lemma \ref{relation of Theta and L alpha star} and the definition of the $\alpha$-level set $L_\alpha^*$, we have
\[
\mathcal{D}^{(s)}_{\Div} \subseteq \bigcup_{\alpha \in [s,1]} L_\alpha^* \subseteq [0,1] \setminus L_0^*.
\]
But $L_0^*$ if of full Lebesgue measure on $[0,1]$ by Theorem \ref{lambda star is zero a.e.}. Therefore, $\mathcal{D}^{(s)}_{\Div}$ has Lebesgue measure zero.
\end{proof}

\begin{proof} [Proof of Theorem \ref{Theta hdim theorem}]
Note first that $\mathcal{D}_{\Div}^{(1)} \subseteq L_1^*$ by Lemma \ref{relation of Theta and L alpha star}. But $\hdim L_1^* = 1 - 1 = 0$ by Theorem \ref{hdim L alpha star theorem}, and therefore, by monotonicity of the Hausdorff dimension (see \cite[p.~48]{Fal14}), we conclude that $\hdim \mathcal{D}_{\Div}^{(1)} = 0$.

Now, fix $s \in (0,1)$. Let $t \in (s,1)$ be arbitrary. Then, by Lemma \ref{relation of Theta and L alpha star}, we have the following inclusions:
\[
L_{t}^* \subseteq \mathcal{D}^{(s)}_{\Div} \subseteq \bigcup_{\alpha \in [s,1]} L_\alpha^*.
\]
But then,
\[
1 - t = \hdim L_{t}^* \leq \hdim \mathcal{D}^{(s)}_{\Div} \leq \hdim \bigcup_{\alpha \in [s,1]} L_\alpha^* = 1 - s,
\]
where we used monotonicity of the Hausdorff dimension (see \cite[p.~48]{Fal14}) for both inequalities and \eqref{hdim formula} for both equalities. On letting $t \to s^+$, we conclude that $\hdim \mathcal{D}^{(s)}_{\Div} = 1-s$, as required.
\end{proof}

\begin{proof}[Proof of Theorem \ref{Theta dense G delta theorem}]
Let $s \in (0,1]$. We make use of the fact that $\Sigma_\infty$ and $\mathbb{I}$ are homeomorphic via the homeomorphism $\varphi|_{\Sigma_\infty} \colon \Sigma_\infty \to \mathbb{I}$ (Proposition \ref{f is homeo}(\ref{f is homeo 3})). Put
\[
L_\infty (\Psi^{(s)}) \coloneqq \{ \sigma \in \Sigma : \Psi^{(s)} (\sigma) = \infty \},
\]
as in Lemma \ref{density lemma 1}. Since $\mathcal{D}^{(s)}_{\Div} \subseteq \mathbb{I}$ by Lemma \ref{relation of Theta and L alpha star}, we have
\begin{align*}
\mathcal{D}^{(s)}_{\Div}
&= \left\{ x \in \mathbb{I} : \sum_{n \in \mathbb{N}} \frac{1}{(d_n(x))^s} \text{ diverges} \right\} 
= \{ x \in \mathbb{I} : \Psi^{(s)} (f(x)) = \infty \} \\
&= \{ \varphi|_{\Sigma_{\infty}} (\sigma) \in \mathbb{I} :  \Psi^{(s)} (\sigma) = \infty \}
= \varphi|_{\Sigma_\infty} (L_\infty (\Psi^{(s)})).
\end{align*}
But we know that $L_\infty (\Psi^{(s)})$ is dense $G_\delta$ in $\Sigma_\infty$ by Lemma \ref{density lemma 1}. Then, the denseness and $G_\delta$-ness of $\mathcal{D}^{(s)}_{\Div}$ in $\mathbb{I}$ follow by the homeomorphism $\varphi|_{\Sigma_\infty} \colon \Sigma_\infty \to \mathbb{I}$. Clearly, $\mathbb{I} = \bigcap_{q \in \mathbb{Q} \cap [0,1]} ([0,1] \setminus \{ q \})$ is dense $G_\delta$ in $[0,1]$. Therefore, $\mathcal{D}^{(s)}_{\Div}$ is dense $G_\delta$ in $[0,1]$; consequently, in view of Fact \ref{basic facts}(\ref{basic facts 1}), we conclude that $\Theta^{(s)}$ is comeager and nonmeager in $[0,1]$; consequently, in view of Fact \ref{basic facts}(1), we conclude that $\mathcal{D}^{(s)}_{\Div}$ is comeager and nonmeager in $[0,1]$.

Now, since $[0,1]$ clearly has no isolated points, Fact \ref{basic facts}(\ref{basic facts 2}) tells us that $\mathcal{D}^{(s)}_{\Div}$ is uncountable. Therefore, in view of Proposition \ref{cardinality of uncoutable Polish}, since $[0,1]$ is Polish, we infer that its uncountable $G_\delta$ subset $\mathcal{D}^{(s)}_{\Div}$ has cardinality $\mathfrak{c}$.
\end{proof}

\begin{proof} [Proof of Corollary \ref{bdim Theta corollary}]
Let $s \in (0,1]$. By Theorem \ref{Theta dense G delta theorem}, we know that $\mathcal{D}^{(s)}_{\Div}$ is dense in $[0,1]$. Therefore, $\bdim \mathcal{D}^{(s)}_{\Div} = 1$ by Proposition \ref{dense bdim proposition}.
\end{proof}

Recall the definitions of two sets $\mathcal{D}_{\Div}$ and $\mathcal{T}_{\Div}$ given in \eqref{definition of E and F}.

\begin{lemma} \label{E and F are equal}
We have $\mathcal{D}_{\Div} = \mathcal{T}_{\Div}$.
\end{lemma}

\begin{proof}
Let $x \in [0,1]$. Note first that since $f(x) \in \Sigma$, we have $d_{n+1}(x) \geq d_n(x)+1$ for each $n \in \mathbb{N}$. Then, by the definitions \eqref{Pierce algorithm 1} and \eqref{Pierce algorithm 2} and the expansion \eqref{Pierce expansion}, exactly one of the following holds:
\[
\begin{dcases}
T^n(x) = 0 = \frac{1}{d_{n+1}(x)}, &\text{if } d_{n+1}(x) = \infty; \\
0 < T^n(x) = \frac{1}{d_{n+1}(x)}, &\text{if } d_{n+1}(x) < \infty \text{ and } d_{n+2}(x) = \infty; \\
0 < \frac{1}{d_{n+1}(x)+1} \leq T^n(x) < \frac{1}{d_{n+1}(x)}, &\text{if } d_{n+1}(x) < \infty \text{ and } d_{n+2}(x) < \infty.
\end{dcases}
\]
Hence, in any case, the inequalities
\[
\frac{1}{2d_{n+1}(x)} \leq \frac{1}{d_{n+1}(x)+1} \leq T^n(x) \leq \frac{1}{d_{n+1}(x)}
\]
hold true, where the first inequality follows by noting that the $d_n(x)$ are $\mathbb{N}_\infty$-valued for each $n \in \mathbb{N}$ (see \eqref{Pierce algorithm 1} and \eqref{Pierce algorithm 2}). Therefore,
\[
\sum_{n \in \mathbb{N}} T^n(x) < \infty \iff \sum_{n \in \mathbb{N}} \frac{1}{d_{n+1}(x)} < \infty \iff \sum_{n \in \mathbb{N}} \frac{1}{d_n(x)} < \infty,
\]
and this completes the proof.
\end{proof}

Observe that, by definition, we have
\begin{align} \label{E equals Theta one}
\mathcal{D}_{\Div} = \left\{ x \in [0,1] : \sum_{n \in \mathbb{N}} \frac{1}{(d_n(x))^1} \text{ diverges} \right\} = \mathcal{D}^{(1)}_{\Div}.
\end{align}

For Corollaries \ref{sum of reciprocals hdim corollary}, \ref{sum of reciprocals dense G delta corollary}, and \ref{Shallit's theorem}, it is enough to prove the results for $\mathcal{D}_{\Div}$ only, since $\mathcal{D}_{\Div} = \mathcal{T}_{\Div}$ by Lemma \ref{E and F are equal}.

\begin{proof} [Proof of Corollary \ref{sum of reciprocals hdim corollary}]
By using \eqref{E equals Theta one} and Theorem \ref{Theta hdim theorem}, we find that $\hdim \mathcal{D}_{\Div} = \hdim \mathcal{D}_{\Div}^{(1)} = 1 - 1 = 0$.
\end{proof}

\begin{proof} [Proof of Corollary \ref{sum of reciprocals dense G delta corollary}]
By Fact \ref{basic facts}(1), it suffices to show that $\mathcal{D}_{\Div}$ is a dense $G_\delta$ subset of $[0,1]$. Since $\mathcal{D}_{\Div} = \mathcal{D}_{\Div}^{(1)}$ by \eqref{E equals Theta one}, and since we know from Theorem \ref{Theta dense G delta theorem} that $\mathcal{D}_{\Div}^{(1)}$ is dense $G_\delta$ in $[0,1]$, the corollary follows.
\end{proof}

\begin{proof}[Proof of Corollary \ref{Shallit's theorem}]
Due to Corollary \ref{sum of reciprocals dense G delta corollary}, it is sufficient to show that $\mathcal{D}_{\Div}$ is Lebesgue-null and has cardinality $\mathfrak{c}$. Since $\mathcal{D}_{\Div}^{(1)}$ has Lesbesgue measure zero by Theorem \ref{Theta Lebesgue measure theorem} and has cardinality $\mathfrak{c}$ by Theorem \ref{Theta dense G delta theorem}, the desired results follow from \eqref{E equals Theta one}.
\end{proof}

\section*{Acknowledgements}

I would like to express my gratitude to my advisor, Dr.\ Hanfeng Li, for his insightful comments during our weekly meetings. I also extend my sincere appreciation to the referees for their thorough review of the manuscript and their thoughtful suggestions, which have not only improved this paper but also provided valuable direction for future research, including potential developments toward a unified theory in real number representation systems.

\end{document}